\documentclass[12pt]{amsart}
\usepackage{amsmath,amsthm,amssymb,euscript,times}

\usepackage[matrix,arrow,curve,cmtip]{xy}

\xymatrixcolsep{1.9pc}                     
\xymatrixrowsep{1pc}
\newdir{ >}{{}*!/-5pt/\dir{>}}             
\newdir{ |}{{}*!/-5pt/\dir{|}}             

\newtheorem{thm}{Theorem}[section]
\theoremstyle{definition}
\newtheorem{defn}[thm]{Definition}
\newtheorem{prop}[thm]{Proposition}
\newtheorem{cor}[thm]{Corollary}
\newtheorem{lemma}[thm]{Lemma}
\newtheorem{rem}[thm]{Remark}
\newtheorem{exam}[thm]{Example}

\newcommand{\cate}{\mathcal}
\newcommand{\topos}{\EuScript}
\newcommand{\cA}{\cate A}
\newcommand{\B}{\cate B}

\newcommand{\D}{\cate D}

\newcommand{\K}{\cate K}
\newcommand{\LL}{\topos L}

\newcommand{\CC}{\mathbb C}
\newcommand{\NN}{\mathbb N}
\newcommand{\QQ}{\mathbb Q}
\newcommand{\ZZ}{\mathbb Z}
\newcommand{\cS}{\cate S}

\newcommand{\uu}{\mathbf{u}}
\newcommand{\vv}{\mathbf{v}}
\newcommand{\xx}{\mathbf{x}}
\newcommand{\zz}{\mathbf{z}}
\newcommand{\X}{\cate X}

\newcommand{\la}{\leftarrow}
\newcommand{\lla}{\longleftarrow}
\newcommand{\ra}{\rightarrow}
\newcommand{\dra}{\Rightarrow}
\newcommand{\lra}{\longrightarrow}
\newcommand{\leri}{\leftrightarrows}
\newcommand{\inc}{\hookrightarrow}
\newcommand{\linc}{\hookleftarrow}
\newcommand{\llra}[1]{\stackrel{#1}{\lra}}	
\newcommand{\llla}[1]{\stackrel{#1}{\lla}}	
\newcommand{\xra}[1]{\xrightarrow{#1}}		
\newcommand{\into}{\rightarrowtail}
\newcommand{\linto}{\leftarrowtail}
\newcommand{\epi}{\twoheadrightarrow}
\newcommand{\lepi}{\twoheadleftarrow}
\newcommand{\biimp}{\Longleftrightarrow}

\DeclareMathOperator{\colim}{colim\,}
\DeclareMathOperator{\et}{\acute{e}t}
\DeclareMathOperator{\Mod}{Mod}
\DeclareMathOperator{\card}{card}
\DeclareMathOperator{\im}{im}
\DeclareMathOperator{\id}{id}

\newcommand{\str}{{\mathsf{str}}}
\newcommand{\emb}{\mathsf{emb}}
\newcommand{\Grp}{\mathit{Grp}}
\newcommand{\Set}{\mathit{Set}}
\newcommand{\Field}{\mathit{Field}}
\newcommand{\Ab}{\mathit{Ab}}
\newcommand{\FreeAb}{\mathit{FreeAb}}
\newcommand{\mono}{\mathsf{mono}}
\newcommand{\ma}{\mathsf{max}}

\renewcommand{\leq}{\leqslant}
\renewcommand{\geq}{\geqslant}

\begin{document}

\title{Elementary equivalences and accessible functors}
\author{T. Beke and J. Rosick\'y$^{*}$} 

\date{\today}

\thanks{$^{*}$ Supported by the Grant agency of the Czech republic under the grant 
               P201/12/G028.}						
\begin{abstract}
We introduce the notion of $\lambda$-equivalence and $\lambda$-embeddings of objects in suitable categories.  This notion specializes to $\LL_{\infty\lambda}$-equivalence and $\LL_{\infty\lambda}$-elementary embedding for categories of structures in a language of arity less than $\lambda$, and interacts well with functors and $\lambda$-directed colimits.  We recover and extend results of Feferman and Eklof on ``local functors'' without fixing a language in advance.  This is convenient for formalizing Lefschetz's principle in algebraic geometry, which was one of the main applications of the work of Eklof.
\end{abstract}
\maketitle


\section*{Introduction}
Elementary equivalence seems to be an intrinsically syntactic notion, and thus an unlikely subject for categorical model theory.  Categorical model theory, after all, focuses on the properties of models describable via a corresponding category of models, independent of the underlying language.  Karp's theorem, however, equates $\LL_{\infty\lambda}$-equivalence of structures with the existence of a set of partial isomorphisms satisfying certain extension properties.  This paper will introduce a calculus of equivalence of objects in suitable categories that specializes to $\LL_{\infty\lambda}$-equivalence for categories of structures.

Our interest was drawn to this problem by papers of Feferman~\cite{F}, Eklof~\cite{E1} and Hodges~\cite{H} from the 1970's, focused on operations on structures that preserve $\LL_{\infty\lambda}$-equivalence.  A beautiful application was Eklof's formalization in \cite{E2} of the Lefschetz principle, namely, that there is ``only one'' algebraic geometry over any universal domain, that is, algebraically closed field of infinite transcendence degree in a given characteristic.  Since any two universal domains of the same characteristic are $\LL_{\infty\omega}$-equivalent, it is natural to take the meaning of ``only one'' to be ``same, up to $\LL_{\infty\omega}$-equivalence'' and to characterize families of operations preserving this equivalence.  The approach of Hodges is syntactic: essentially, he defines what is an allowable passage from structures to structures via a transfinite sequence of extensions by definition.  The work of Feferman and Eklof, which is very close in spirit to ours, augments the syntactic analysis with the observation that the passage should be \textsl{functorial}.

What is common to this circle of papers is that the underlying objects are taken to be structures in a specific language, with respect to which $\LL_{\infty\omega}$-equivalence, and ``underlying sets'', should be understood.  Here, we introduce a notion of $\lambda$-equivalence of objects, denoted $\sim_\lambda$, in purely category-theoretic terms.  It is, essentially, back-and-forth equivalence satisfying the extension property with respect to objects of size less than $\lambda$, where --- in the categories of interest to us --- a ``size'' for objects can be defined intrinsically.  $\lambda$-equivalence interacts well with functors preserving $\lambda$-directed colimits, and specializes to the theory of Feferman--Eklof over structures.

Section~\ref{monogen} will furnish all the details, but even in the absence of precise definitions, the following example should illustrate aspects of our work.  The statements below follow from Example~\ref{set}, Prop.~\ref{pres}, Thm.~\ref{eklof} and Prop.~\ref{simiso}.
\begin{itemize}
\item Let $\lambda$ be an infinite regular cardinal.  For sets $U$ and $V$, as objects of the category $\Set_\mono$ of sets and injective functions, $U\sim_\lambda V$ iff either $|U|=|V|<\lambda$, or both $|U|\geq\lambda$ and $|V|\geq\lambda$.
\item Fix a field $k$, and let $\Field_k$ be the category of fields and homomorphisms over $k$.  Let $F:\Set_\mono\ra\Field_k$ be a functor that sends $U$ to an algebraic closure of $k(U)$, the purely transcendental extension of $k$ on the transcendence basis $U$.  If $U\sim_\lambda V$ in $\Set_\mono$ then $F(U)\sim_\lambda F(V)$ in $\Field_k$.
\item Let $p$ be a prime distinct from the characteristic of $k$, and let $X_k$ be a reduced scheme of finite type over $k$.  For a field $K$ over $k$, let $X_K$ denote the base extension of $X_k$ to $K$.  Fix $n\in\NN$ and let $H:\Field/k\ra\Ab$ be the functor that sends $K$ to $H^n_{\et}(X_K,\ZZ/p\ZZ)$, the $n^\textup{th}$ \'{e}tale cohomology group of $X_K$ with constant coefficients $\ZZ/p\ZZ$.  $H$ preserves filtered colimits.  Hence, if $K_1\sim_\lambda K_2$ in $\Field/k$ then $H^n_{\et}(X_{K_1},\ZZ/p\ZZ)\sim_\lambda H^n_{\et}(X_{K_1},\ZZ/p\ZZ)$ in the category of abelian groups and homomorphisms.
\item If $A\sim_\lambda B$ in $\Ab$ and one of these groups is finitely generated, then they are isomorphic.
\end{itemize}
Let $\lambda=\omega$.  The upshot is that the finite generation of $H^n(X_K,\ZZ/p\ZZ)$ for one algebraically closed field $K$ of infinite transcendence degree over $\QQ$ --- say, for the complex numbers --- implies that these cohomology groups are finitely generated for any such $K$, and indeed, (non-canonically)\ isomorphic to their value for $K=\CC$.  With Thm.~\ref{emb-pres} in place of Thm.~\ref{eklof}, one can then prove that any inclusion $K_1\inc K_2$ between algebraically closed fields of infinite transcendence degree over the prime field, induces a homomorphism $H^n_{\et}(X_{K_1},\ZZ/p\ZZ)\ra H^n_{\et}(X_{K_1},\ZZ/p\ZZ)$ that is an $\omega$-embedding in the category of abelian groups, hence an isomorphism.  These statements have just the form predicted by the Lefschetz principle.

The extra-logical input --- namely, functoriality of \'{e}tale cohomology, its preservation of suitable colimits, and its behavior over fields such as the complex numbers, where topological tools are available --- is crucial, of course.  (Nor are the results new: they are well-known consequences of the proper base change theorem.)  The point is that the notions of $\lambda$-equivalence and $\lambda$-embedding form a good fit with the `natural language' of algebraic geometry that employs functors, natural transformations, isomorphisms and generating ranks of algebraic objects etc.  This is not to say that it is impossible to encode algebraic geometry within a first-order language fixed in advance, but it is a delicate matter indeed, especially the seemingly second-order structures associated with Grothendieck sites, sheaves and derived functors.  The reader should consult the second part of Eklof~\cite{E2} for an `encoding' of a much simpler situation.

The notion of $\sim_\lambda$ within, say, the category $\Field_k$ is insensitive as to what language one chooses to axiomatize fields with: one can reason with mathematical objects directly.  $\lambda$-equivalence of fields \textsl{is}, however, sensitive as to what is taken to be a morphism of fields.  In fact, the internal notion of ``size'' depends not only on objects, but on the class of morphisms chosen: the input to the machinery of $\lambda$-equivalence is a \textsl{category}.  This category is not quite arbitrary; for $\lambda$-equivalence to work well, the category has to satisfy the property that every object can be written as a $\lambda$-directed colimit of $\lambda$-generated subobjects.  (See the meaning of $\lambda$-generated below.)  This condition follows from the axioms of Abstract Elementary Classes, and also from those of accessible categories.  It is hard, in fact, to come up with a situation where one can do meaningful infinitary model theory where this condition is not satisfied for some $\lambda$.

We plan to return to applications of $\lambda$-equivalence to the Lefschetz--Weil principle, and compare our approach with those of Feferman, Eklof of Hodges.  This paper is devoted to the fundamentals of $\sim_\lambda$: the categories we are concerned with, spans and equivalences, elementary embeddings, and elementary chains.  Many of the proofs, free of underlying sets, are combinatorial arguments with diagrams in categories satisfying Def.~\ref{monogendef}.  It is a pleasant surprise that one can do so much, starting with so little.

\section{Mono-generated categories}  \label{monogen}
\begin{defn}  \label{monogendef}
Let $\lambda$ be a regular cardinal and $\cA$ a category.  An object $X$ of $\cA$ is \textsl{$\lambda$-generated} if $\hom(X,-):\cA\ra\Set$ preserves those $\lambda$-directed colimits of monomorphisms that exist in $\cA$.  The category $\cA$ is \textsl{$\lambda$-mono-generated} provided every object of $\cA$ can be written as a colimit of a $\lambda$-directed diagram consisting of monomorphisms and $\lambda$-generated objects.  A category is \textit{mono-generated} if it is $\lambda$-mono-generated for some regular cardinal $\lambda$.
\end{defn}

Note that it is not assumed that a $\lambda$-mono-generated category has all $\lambda$-directed colimits of monos, though we will occasionally need this as a separate assumption.  Functors $F:\cA\ra\B$ between $\lambda$-mono-generated categories will typically be assumed to preserve those $\lambda$-directed colimits of monos that exist in $\cA$.  Mono-generated should probably be thought of as the weakest assumption on a category that allows one to `approximate' an arbitrary object $X$ by `small' subobjects, where the sense of approximation ($\lambda$-directed colimit)\ and smallness (being $\lambda$-generated)\ are defined in terms of the category itself.

\begin{lemma}  \label{monococone}
Let $\cA$ be a $\lambda$-mono-generated category and $\D$ a $\lambda$-directed diagram.  Below, we will assume that the colimits mentioned actually exist.

$(i)$ If $D:\D\ra\cA$ is a functor taking values in monomorphisms then the components $k_d:D(d)\to\colim D$ of its colimit cocone are monomorphisms.  $(ii)$ If $D_1,D_2:\D\ra\cA$ are functors taking values in monomorphisms, and $\eta:D_1\ra D_2$ is a natural transformation such that $\eta(d)$ is a monomorphism for all $d\in\D$, then the induced map $m:\colim D_1\ra\colim D_2$ is mono.  $(iii)$ Let $D$ be as in part $(i)$.  If $E$ is the target of a cocone consisting of monomorphisms on $D$, then the induced map $e:\colim D\ra E$ is mono.
\end{lemma}

\begin{proof}
$(i)$ It suffices to check that $k_df=k_dg$ implies $f=g$ for any $f,g:G\to D(d)$ with $\lambda$-generated $G$.  But indeed, since $G$ is $\lambda$-generated and $\D$ is $\lambda$-directed, $D(h)f=D(h)g$ for some $h:d\to d'$ in $\D$, whence $f=g$.  $(ii)$ Let $G$ be $\lambda$-generated and let maps $f,g:G\ra\colim D_1$ be given such that $mf=mg$.  We wish to prove $f=g$.  Find $d\in\D$ so that $f$ and $g$ factor respectively as $G\xra{f_0,\,g_0}D_1(d)\ra\colim D_1$.  The composites
\[  G \xra{f_0,\,g_0} D_1(d) \llra{\eta(d)} D_2(d) \ra \colim D_2 \]
are both equal to $mf=mg$.  Since $\eta(d)$ is mono and $D_2(d)\ra\colim D_2$ is mono by $(i)$, $f_0=g_0$ so $f=g$.  $(iii)$  Apply $(ii)$ with $D_1=D$, $D_2$ the functor that is constant on $E$, and $\eta$ the cocone.
\end{proof}

\begin{cor}  \label{monomono}
Let $\cA$ be a $\lambda$-mono-generated category and let $\cA_\mono$ be the subcategory of $\cA$ with the same objects, but only the monos as morphisms.  Then $\cA_\mono$ is $\lambda$-mono-generated.
\end{cor}

Indeed, by the lemma the inclusion $\cA_\mono\inc\cA$ creates $\lambda$-directed colimits, and an object of $\cA$ is $\lambda$-generated in $\cA$ if and only if it is $\lambda$-generated in $\cA_\mono$.  \qed

The next observation is the analogue of ``raising the index of accessibility'' in the context of mono-generated categories.  Recall the relation $\triangleleft$ of ``sharply less than'' and its properties from Makkai--Par\'{e}~\cite{MP}~2.3.

\begin{prop}  \label{raise}
Suppose $\cA$ is a $\lambda$-mono-generated categories possessing colimits of $\lambda$-directed diagrams of monos.  If $\lambda\triangleleft\kappa$, then $\cA$ is $\kappa$-mono-generated.
\end{prop}

\begin{proof}
Write an object $X$ as the colimit of a $\lambda$-directed diagram consisting of monomorphisms and $\lambda$-generated objects. Let $\D_\kappa$ be the poset formed by $\lambda$-directed subdiagrams of $\D$ of size less than $\kappa$.  By \cite{MP}~Cor.~2.3.9, $\D_\kappa$ is $\kappa$-directed, and the colimit of the functor $\D_\kappa\ra\cA$ that sends each $\lambda$-directed subdiagram of $\D$ of size less than $\kappa$, to its colimit, is isomorphic to $X$. Since
colimits of diagrams of size less than $\kappa$ consisting of monomorphisms and $\lambda$-generated objects are $\kappa$-generated (the proof is analogous
to \cite{AR}, 1.16), $\cA$ is $\kappa$-mono-generated. 
\end{proof}

Mono-generated is a weakening of the notion of \textsl{accessible category}, introduced by Makkai and Par\'{e} in ~\cite{MP}.  The rest of this section (which may be skipped on first reading)\ concerns the relation between accessibility, mono-generatedness, and its variants.

Recall that a category $\cA$ is \textsl{$\lambda$-accessible}, where $\lambda$ is a regular cardinal, provided that
\begin{enumerate}
\item[($a$)] $\cA$ has $\lambda$-filtered colimits and
\item[($b$)] a set $\X$ of $\lambda$-presentable objects such that every object of $\cA$ is a $\lambda$-filtered colimit of objects from $\X$.
\end{enumerate}
Here, an object $X$ is called $\lambda$-\textit{presentable} if $\hom(X,-):\cA\to\Set$ preserves $\lambda$-filtered colimits.  A functor $F:\cA\ra\B$ between $\lambda$-accessible categories is called $\lambda$-accessible if it preserves $\lambda$-filtered colimits.  See the monographs of Makkai--Par\'{e}~\cite{MP} or Ad\'{a}mek--Rosick\'{y}~\cite{AR} for detailed information on the (2-)category of accessible categories.

Let us call $\cA$ \textsl{$\lambda$-mono-accessible} if
\begin{enumerate}
\item[($a'$)] $\cA$ has $\lambda$-directed colimits of monomorphisms and
\item[($b'$)] a set $\X$ of $\lambda$-generated objects such that every object of $\cA$ is a $\lambda$-directed colimit of a diagram consisting of monomorphisms and objects from $\X$.
\end{enumerate}
This notion was introduced in Ad\'{a}mek--Rosick\'{y}~\cite{AR1}, motivated by the \textsl{locally generated} categories of Gabriel and Ulmer.

Chorny and Rosick\'{y}~\cite{CR} investigated the consequences of omitting the `set' clause in condition ($b$).  They call $\cA$ \textsl{$\lambda$-class-accessible} if
\begin{enumerate}
\item[($a$)] $\cA$ has $\lambda$-filtered colimits and
\item[($b''$)] every object of $\cA$ can be written as a $\lambda$-filtered colimit of $\lambda$-presentable objects.
\end{enumerate}

Let us point out some relations between these notions:

(1) Every locally $\lambda$-presentable category (that is, cocomplete $\lambda$-accessible category)\ is locally generated, and mono-accessible; see~\cite{AR}~1.70 and \cite{AR1}.

(2) Every accessible category $\cA$ is mono-accessible.  Note, however, that a $\lambda$-accessible $\cA$ may be $\kappa$-mono-accessible only for certain $\kappa\geq\lambda$.  Indeed, consider the full subcategory of the category of morphisms of $\cA$ whose objects are the monomorphisms of $\cA$.  By Prop.~6.2.1 of Makkai-Par\'{e}~\cite{MP}, this category is accessible, hence monomorphisms of $\cA$ are closed under $\kappa_1$-filtered colimits for some $\kappa_1$.  By Theorem~2.34 of Ad\'{a}mek-Rosick\'{y}~\cite{MP}, the subcategory of $\cA$ consisting of $\lambda$-pure monomorphisms is accessible.  The proof shows, in particular, that every object $X$ is the $\kappa_2$-directed union of $\kappa_2$-presentable pure subobjects of $X$, for some $\kappa_2$.  Thence $\cA$ is $\max\{\kappa_1,\kappa_2\}$-mono-accessible.

(3) A mono-accessible category need not be accessible.  For example, let $\cA$ be a mono-accessible category and add freely an idempotent endomorphism $f:X\ra X$ to an object $X$ of $\cA$.  Let $\cA_+$ be the resulting category.  Since $\cA_+$ has the same monomorphisms as $\cA$, it is mono-accessible, but the idempotent $f$ does not split in $\cA_+$.  In an accessible category, however, every idempotent splits, cf.~\cite{AR}~2.4.

(4) If all morphisms of a category $\cA$ are monomorphisms, then $\cA$ is $\lambda$-accessible if and only if it is $\lambda$-mono-accessible.  An important source of examples are Shelah's Abstract Elementary Classes. (See the monograph of Baldwin~\cite{Ba} for an introduction.)  Let $\K$ be an Abstract Elementary Class, and let $\cA_\K$ be the category whose objects are the structures in $\K$ and whose morphisms are the strong embeddings.  Then $\cA_\K$ is $\lambda^+$-mono-accessible where $\lambda$ is the L\"{o}wenheim-Skolem number of $\K$, cf.\ Lieberman~\cite{L}.

The mono-accessible categories arising this way are special; for example, they have directed colimits preserved by some functor into $\Set$.  See Beke-Rosick\'{y}~\cite{BR} for category-theoretic characterizations of AEC.

(5) Let $\cA$ be a class-accessible but not accessible category, for example the free $\kappa$-filtered cocompletion of a large category, or the category of presheaves on a large category.  Then $\cA_\mono$ is mono-generated but not mono-accessible.

(6) Consider $\FreeAb_\mono$, the full subcategory of the category $\Ab_\mono$ of abelian groups and monomorphisms, with objects the free abelian groups.  Let $F:\D\ra\FreeAb_\mono$ be a filtered diagram.  Suppose that $X=\colim F$ exists.  Let $Y$ be the colimit of the composite $F:\D\ra\FreeAb_\mono\ra\Ab_\mono$.  There is a natural monomorphism $Y\ra X$; since $X$ is a free abelian group, so is $Y$.  So $Y$ has the universal mapping property in $\FreeAb_\mono$ as well; since colimits are unique up to isomorphism, $X$ is isomorphic to $Y$.  So filtered colimits in $\FreeAb_\mono$ (to the extent they exist)\ are ``standard'', i.e.\ created by the inclusion $\FreeAb_\mono\ra\Ab_\mono$.  It follows that any finitely generated free abelian group is $\omega$-generated as an object of $\FreeAb_\mono$.  Since any free abelian group is, canonically, the directed union of its finitely generated subgroups, $\FreeAb_\mono$ is an $\omega$-mono-generated category.

By contrast, the accessibility $\FreeAb_\mono$ depends on set theory.  In the constructible universe, $\FreeAb_\mono$ is not accessible, so it is consistent (relative to ZFC)\ that $\FreeAb_\mono$ is not accessible.  But if $\kappa$ is a compact cardinal, then the $\kappa$-filtered colimit of free abelian groups is free, and $\FreeAb_\mono$ is $\kappa$-accessible, cf.\ Eklof--Mekler~\cite{EM}.

Let $F:\K_1\ra\K_2$ be a $\kappa$-accessible functor between $\kappa$-accessible categories, all of whose morphisms are monos.  Then a similar conclusion holds for the \textsl{powerful image} of $F$, cf.\ Makkai--Par\'{e}~\cite{MP}~5.5.

(7) Finally, for any Abstract Elementary Class in Shelah's sense, the class of structures and strong embeddings form an accessible category.  The relation between AEC's and accessible categories has been investigated by several articles; see, for example, Beke--Rosick\'{y}~\cite{BR}.

All in all, one has logical implications
\[ \xymatrix{ \textup{AEC} \ar[r] & \textup{accessible} \ar[r]\ar[d] & \textup{mono-accessible} \ar[d] \\
& \textup{class-accessible} \ar[r] & \textup{mono-generated} } \]
none of which is reversible.

Observe that the passage from $\lambda$-mono-accessible categories to $\lambda$-mono-generated ones is the same as the passage from accessible categories to preaccessible ones, in the sense of Ad\'{a}mek--Rosick\'{y}~\cite{AR2}.

\section{Spans and equivalences}
Let $\cA$ be a category, $X$ and $Y$ objects of $\cA$.  Recall that a category-theoretic span between $X$ and $Y$ is a diagram of the form $X\la U\ra Y$.  If the arrow $X\la U$ is mono, this can be thought of as a `partial morphism' from $X$ to $Y$; if both arrows are mono, as a `partial isomorphism', i.e.\ an identification of a subobject of $X$ with a subobject of $Y$.  Some of this work can be developed in the context of category-theoretic spans, but many of our applications are restricted to the case of spans whose arrows are mono.  So as not to introduce cumbersome terminology, we will restrict attention to mono-spans, but call them simply \textsl{spans}.

\begin{defn}
A \textsl{span} between $X$ and $Y$ is a diagram 
\[  X\linto U\into Y  \]
where both arrows are monomorphisms.  When there is no danger of confusion, we will sometimes refer to a span by its center object $U$.  A span $X\linto U\into Y$ will be called a $\lambda$-\textit{span} if $U$ is $\lambda$-generated.  (We will mostly, though not exclusively, use this notion if the ambient category is $\lambda$-mono-generated.)

A morphism $(X\linto U\into Y)\to(X\linto V\into Y)$ of spans is a map $U\to V$ such that the diagram
\[ \xymatrix{ X & V\ar[l]\ar[r] & Y \\
& U \ar[lu]\ar[ru]\ar[u] } \]
commutes.
\end{defn}

\begin{defn}  \label{dense}
Let $\cA$ be a category, $X$ and $Y$ objects of $\cA$.  A set $X\linto U_i\into Y$, $i\in I$ of spans is $\lambda$-\textsl{dense} if it is non-empty and
\begin{itemize}
\item[\textsf{(back)}] for all $i\in I$ and monomorphism $X\linto G$ with $\lambda$-generated $G$, there exist $j\in I$ and morphisms $G\ra U_j$ and $U_i\ra U_j$ such that
\[ \xymatrix{ & U_j \ar[ld]\ar[rd] \\
X & G\ar[l]\ar[u] & Y \\
& U_i \ar[lu]\ar[ru]\ar@/_5mm/[uu] } \]
commutes.
\item[\textsf{(forth)}] for all $i\in I$ and monomorphism $G\into Y$ with $\lambda$-generated $G$, there exist $j\in I$ and morphisms $G\ra U_j$ and $U_i\ra U_j$ such that
\[ \xymatrix{ & U_j \ar[ld]\ar[rd] \\
X & G\ar[r]\ar[u] & Y \\
& U_i \ar[lu]\ar[ru]\ar@/^5mm/[uu] } \]
commutes.
\end{itemize}
\end{defn}

\begin{defn}  \label{equi}
Let $X$ and $Y$ be objects of a category $\cA$.  We will say that $X$ and $Y$ are $\lambda$-equivalent, denoted $X\sim_\lambda Y$, if there exists a $\lambda$-dense set of spans between $X$ and $Y$.
\end{defn}

\begin{rem}  \label{max}
\begin{itemize}
\item[(1)] Obviously, $X\sim_\kappa Y$ implies $X\sim_\lambda Y$ for $\lambda<\kappa$; and if $X$ and $Y$ are isomorphic, then $X\sim_\lambda Y$, via any $X\llla{f}U\llra{g}Y$ where $f,g$ are isomorphisms.  The relation $\sim_\lambda$ is symmetric and reflexive.  We will soon see that it is transitive, hence an equivalence relation, when $\cA$ is $\lambda$-mono-generated.
\item[(2)] Let $\cA$ be a $\lambda$-mono-generated category.  Then $X\sim_\lambda Y$ in $\cA$ if and only if $X\sim_\lambda Y$ in $\cA_\mono$.  This follows from Cor.~\ref{monomono} and the fact that all morphisms in the test diagrams in Def.~\ref{dense} are mono.
\item[(3)] If $\cS_i$, $i\in I$, are $\lambda$-dense sets of spans between $X$ and $Y$, so is their union $\bigcup_{i\in I}\cS_i$.  It follows that if $X\sim_\lambda Y$, then there is a greatest $\lambda$-dense set of spans between $X$ and $Y$.  This greatest $\lambda$-dense set of spans is a sieve in the sense that if $X\linto V\into Y$ belongs to it and
\[  (X\linto U\into Y)\to(X\linto V\into Y)  \]
is a morphism of spans, then $X\linto U\into Y$ belongs to it as well.  Analogously, there exists a greatest $\lambda$-dense set of $\lambda$-spans between $X$ and $Y$.
\end{itemize}
\end{rem}

\begin{exam}  \label{set}
Work in the category $\Set$, and let $X$ and $Y$ be sets.  Suppose $|X|=|Y|$ or both $|X|\geq\lambda$ and $|Y|\geq\lambda$.  Then the set of all $\lambda$-spans between $X$ and $Y$ is $\lambda$-dense.

If $|X|=|Y|$ then the greatest $\lambda$-dense set of spans between $X$ and $Y$ is that of all spans.  If $|X|,|Y|\geq\lambda$ and $|X|\not=|Y|$ then it is
\[  \{ \, X\stackrel{i}{\linto}U\stackrel{j}{\into}Y\textup{\ such that\ }|X-\im(i)|\geq\lambda\textup{\ and\ }|Y-\im(j)|\geq\lambda| \, \}. \]
If, say, $|X|<\lambda$ and $X\sim_\lambda Y$ then the back property establishes the existence of a span $X\linto U\into Y$ in the family where the left arrow is bijective, which contradicts the forth property unless the right arrow is onto.  So, in both $\Set$ and $\Set_\mono$, $X\sim_\lambda Y$ if and only if either $|X|=|Y|<\lambda$ or $\lambda\leq |X|,|Y|$.  This is the same as $X$ and $Y$ being $L_{\infty\lambda}$-equivalent as structures for the language containing only equality.  As we will see in Theorem~\ref{karp}, this is no coincidence.
\end{exam}

\begin{prop}  \label{densevert}
Suppose the category $\cA$ possesses a class $\D$ of objects that are $\lambda$-mono-dense; that is, every object of $\cA$ can be written as a $\lambda$-directed colimit of subobjects that are isomorphic to some element of $\D$.  Suppose $X\sim_\lambda Y$ in $\cA$.  Then there exists a $\lambda$-dense set of spans between $X$ and $Y$ of the form $X\linto V\into Y$ where $V\in\D$.
\end{prop}

\begin{proof}
Let $X\linto U_i\into Y$, $i\in I$ be a $\lambda$-dense set of spans.  Write each $U_i$ as a $\lambda$-directed colimit of subobjects $V_{ij}$, $j\in J_i$, with $V_{ij}\in\D$.  Composing with the structure maps $X\linto U_i\into Y$, this gives a set $X\linto V_{ij}\into Y$, $i\in I$, $j\in J_i$ of spans.  We will show that this set is also $\lambda$-dense.  Given an $X\linto V_{ij}\into Y$ and $G\into X$ with $\lambda$-generated $G$, there exist, by the density of $\{U_i\;|\;i\in I\}$, an $i'\in I$ with morphisms $G\ra U_{i'}$ and $U_i\ra U_{i'}$ such that
\[ \xymatrix{ & U_{i'} \ar[ld]\ar[rd] \\
X & G\ar[l]\ar[u] & Y \\
& U_i \ar[lu]\ar[ru]\ar@/_5mm/[uu] } \]
commutes.  Since $U_{i'}$ is the $\lambda$-directed colimit of $V_{i'j}$, $j\in J_{i'}$, there exists $V_{i'j'}$, $j'\in J_{i'}$ such that both $G\to U_{i'}$ and $V_{ij}\to U_i\to U_{i'}$ factorize through $V_{i'j'}$.  The span $X\linto V_{i'j'}\into Y$ now verifies the back property of Def.~\ref{dense}.  The forth part is analogous.
\end{proof}

\begin{cor}  \label{lambda}
Let $\cA$ be a $\lambda$-mono-generated category.  Then $X\sim_\lambda Y$ if and only if there exists a $\lambda$-dense set of $\lambda$-spans between $X$ and $Y$.
\end{cor}

\begin{prop}
Let $X\sim_\lambda Y$ in a $\lambda$-mono-generated category.  The following are equivalent:
\begin{itemize}
\item[(i)] the greatest $\lambda$-dense set of spans between $X$ and $Y$ contains all $\lambda$-spans between $X$ and $Y$
\item[(ii)] the set of all $\lambda$-spans between $X$ and $Y$ is $\lambda$-dense.
\end{itemize}
\end{prop}

\noindent
(ii)$\dra$(i)\ is a tautology, while the proof of Prop~\ref{densevert}, applied to the greatest $\lambda$-dense set of spans between $X$ and $Y$, shows (i)$\dra$(ii). \qed

Given a set $\cS_{XY}$ of $\lambda$-spans between $X$ and $Y$, and a set $\cS_{YZ}$ of $\lambda$-spans between $Y$ and $Z$, let $\cS_{XY}\star\cS_{YZ}$ denote the set of $\lambda$-spans $X\linto U\into Z$ that can be factored through a commutative diagram of the form
\begin{equation}   \tag{$\dag$}
\xymatrixcolsep{1pc}\xymatrixrowsep{1pc}
\xymatrix{ X && Y && Z \\
& S\ar[ul]\ar[ur] && T\ar[ul]\ar[ur] \\
&& U \ar[ul]\ar[ur] }
\end{equation}
with $X\linto S\into Y\in\cS_{XY}$ and $Y\linto T\into Z\in\cS_{YZ}$.

\begin{prop}  \label{compose}
If $\cS_{XY}$ and $\cS_{YZ}$ are $\lambda$-dense sets of $\lambda$-spans, then $\cS_{XY}\star\cS_{YZ}$ is a $\lambda$-dense set of $\lambda$-spans too.
\end{prop}

\begin{proof}
First of all, $\cS_{XZ}$ is non-empty.  Indeed, given an arbitrary
\[ \xymatrixcolsep{1pc}\xymatrixrowsep{1pc}
\xymatrix{ X && Y && Z \\
& S\ar[ul]\ar[ur] && T\ar[ul]\ar[ur] } \]
with $X\linto S\into Y\in\cS_{XY}$, $Y\linto T\into Z\in\cS_{YZ}$, by density one can find $Y\linto T_0\into Z\in\cS_{YZ}$ and arrow $S\ra T_0$ making
\[ \xymatrixcolsep{1pc}\xymatrixrowsep{1pc}
\xymatrix{ X && Y & T_0\ar[l]\ar[r] & Z \\
& S\ar[ul]\ar[ur]\ar[rru]|!{[ur];[rr]}\hole && T\ar[ul]\ar[ur]\ar[u] } \]
commutative.  But that means
\[ \xymatrixcolsep{1pc}\xymatrixrowsep{1pc}
\xymatrix{ X && Y && Z \\
& S\ar[ul]\ar[ur] && T_0\ar[ul]\ar[ur] \\
&& S \ar@{=}[ul]\ar[ur] } \]
belongs to $\cS_{XZ}$.

Now consider ($\dag$)\ and suppose $X\linto G$ with $\lambda$-generated $G$ is given.  Apply density to $X\linto S\into Y$ and $X\la G$ to find appropriate $X\linto S_0\into Y\in\cS_{XY}$ with $S\ra S_0$ and $G\ra S_0$, then apply density to $Y\linto T\into Z$ and $S_0\ra Y$ to find appropriate $Y\linto T_0\into Z\in\cS_{YZ}$ and $S_0\ra T_0$ in
\[ \xymatrix{ X & S_0\ar[l]\ar[r]\ar@/^5mm/[rr] & Y & T_0\ar[l]\ar[r] & Z \\
G \ar[u]\ar[ru]|!{[u];[r]}\hole & S\ar[ul]\ar[ur]\ar[u] && T\ar[ul]\ar[ur]\ar[u] \\
&& U \ar[ul]\ar[ur] } \]
The composite $U\ra T\ra T_0\ra Y$ equals the composite $U\ra S\ra S_0\ra T_0\ra Y$.  Since $T_0\into Y$ is mono, the composite $U\ra T\ra T_0$ equals the composite $U\ra S\ra S_0\ra T_0$.  That implies that $X\linto S_0\into Z$, which belongs to $\cS_{XZ}$ thanks to
\[ \xymatrixcolsep{1pc}\xymatrixrowsep{1pc}
\xymatrix{ X && Y && Z \\
& S_0\ar[ul]\ar[ur] && T_0\ar[ul]\ar[ur] \\
&& S_0 \ar@{=}[ul]\ar[ur] } \]
verifies the ``back'' part of the density condition for $X\linto U\into Z$ and $X\linto G$, via the connecting map $U\ra S\ra S_0$.  (The remaining commutativities are easy to check, and do not require that arrows be mono.)

The ``forth'' case is symmetric.
\end{proof}

\begin{rem}
On $\lambda$-dense sets of $\lambda$-spans, the operation $\star$ is associative, that is,
\[   \big(\cS_{XY}\star\cS_{YZ}\big)\star\cS_{ZW}=\cS_{XY}\star\big(\cS_{YZ}\star\cS_{ZW}\big)  \]
but we will not need this fact.
\end{rem}

\begin{cor}  \label{mono-eq}
Let $\cA$ be a $\lambda$-mono-generated category.  Then the relation $\sim_\lambda$ is transitive.
\end{cor}

Indeed, by Cor.~\ref{lambda}, if $\lambda$-dense sets of spans exist between $X$ and $Y$, and $Y$ and $Z$, then $\lambda$-dense sets of $\lambda$-spans exist as well, and their $\star$-composite verifies that $X\sim_\lambda Z$.  \qed

\begin{prop}   \label{simiso}
Let $\cA$ be a $\lambda$-mono-generated category.  Suppose $X$ is a $\lambda$-generated object and $X\sim_\lambda Y$.  Then $X$ and $Y$ are isomorphic.
\end{prop}

\begin{proof}
Since $X$ itself is $\lambda$-generated, applying the ``back'' direction of density to the identity map $X\la X$, one constructs a span
\[ \xymatrix{  & U \ar[dl]_x \ar[dr] \\
X \ar@{=}[r] & X\ar[u]_u & Y  }  \]
such that $xu=\id_X$. Thus $x$ is both a monomorphism and a split epimorphism, hence it is an isomorphism.  Let $G\into Y$ be any monomorphism with $\lambda$-generated $G$; density implies the existence of
\[ \xymatrix{ & V \ar[ld]_v\ar[rd] \\
X & G\ar[r]\ar[u] & Y \\
& U \ar[lu]^x\ar[ru]_y\ar@/^5mm/[uu]^i } \]
Again, $v$ is both a mono and a split epi, so an isomorphism; thus so is $i$.  This means that $y:U\into Y$ is a mono such that any mono $G\into Y$ with $\lambda$-generated $G$ factors through $y$.  Since $\cA$ is a $\lambda$-mono-generated category, this implies that $y$ is an isomorphism too.
\end{proof}

\begin{cor}  \label{isochar}
Let $X$, $Y$ be objects of a $\mu$-mono-generated category $\cA$ that has colimits of $\mu$-directed diagrams of monos.  Then $X$ and $Y$ are isomorphic if and only if there exist arbitrarily large regular cardinals $\lambda$ such that $X\sim_\lambda Y$; equivalently: if and only if $X\sim_\lambda Y$ for all regular cardinals $\lambda$.
\end{cor}

Only the `if' direction is non-trivial.  Write $X$ as a $\mu$-directed colimit of $\mu$-generated subobjects; if this colimit has size $\kappa$, then $X$ is $\max\{\mu,\kappa^+\}$-generated.  By \cite{MP}~2.3 or \cite{AR}~2.11, there exists a regular cardinal $\lambda$ such that $\mu\triangleleft\lambda$ and $\kappa<\lambda$.  By Prop.~\ref{raise}, $\cA$ is $\lambda$-mono-generated, and $X$ is $\lambda$-generated.  Prop.~\ref{simiso} now applies.

The next proposition, when the domain and codomain of the functor $F$ are categories of structures, specializes to the main result of Feferman~\cite{F}; see also Thm.~\ref{karp}.

\begin{prop}  \label{pres}
Let $\cA$ be a $\lambda$-mono-generated category and $F:\cA\ra\B$ a functor preserving monomorphisms and $\lambda$-directed colimits of monomorphisms.  If $X\sim_\lambda Y$ then $F(X)\sim_\lambda F(Y)$.
\end{prop}

\begin{proof}
Let $X\la U_i\ra Y$, $i\in I$, be a $\lambda$-dense set of spans between $X$ and $Y$.  Then $F(X)\la F(U_i)\ra F(Y)$, $i\in I$ is a set of spans between $F(X)$ and $F(Y)$ and the claim is that this set is $\lambda$-dense.  Indeed, let some $F(X)\la F(U_i)\ra F(Y)$ and $F(X)\linto G$, with $\lambda$-generated $G$, be given.  Write $X$ as a $\lambda$-directed colimit of $\lambda$-generated objects $X_\alpha$ and monomorphisms; thence $F(X)$ is the $\lambda$-directed colimit of the $F(X_\alpha)$ along monomorphisms.  Find $\alpha$ such that $F(X)\linto G$ factors as $F(X)\la F(X_\alpha)\la G$.  Apply density to $X\la U_i\ra Y$ and $X\la X_\alpha$ to find appropriate $X\la U_j\ra Y$, $U_i\ra U_j$ and $X_\alpha\ra U_j$.  The diagram
\[ \xymatrix{  & F(U_j) \ar[dl]\ar[dr] \\
F(X) & F(X_\alpha)\ar[l]\ar[u] & F(Y) \\
& G \ar[u]\ar[ul] \\
& F(U_i) \ar[uul]\ar[uur]\ar@/_9mm/[uuu]
} \]
verifies the ``back'' case of the $\lambda$-density of $F(X)\la F(U_i)\ra F(Y)$, $i\in I$.  The ``forth'' case is symmetric.
\end{proof}

The usefulness of the previous proposition is limited by the need for $F:\cA\ra\B$ to preserve monos.  A number of functors appearing in e.g.\ commutative algebra and algebraic geometry need not preserve monos, although the other assumptions of Prop.~\ref{pres} are commonly satisfied.  So the next variant, motivated by work of Eklof~\cite{E1}, is especially welcome.  In it, the hypothesis that $F$ preserves monos is dropped; however, one needs to assume more structure on the target category $\B$.  Our discussion will follow Makkai-Par\'{e}~\cite{MP}~3.2.

Let $\lambda$ be a regular cardinal and $\Sigma$ a language for the logic $\LL_{\infty\kappa}$, i.e.\ possibly many-sorted, with relation and function symbols of arity less than $\lambda$.  Homomorphisms of $\Sigma$-structures are required to preserve the interpretations of function symbols and the existing relations.  Embeddings are injective homomorphisms $A\ra B$ that preserve and reflect relations, inducing an isomorphism of $A$ with the substructure of $B$ on the set-theoretical image of $A$.  We denote by $\str(\Sigma)$ the category of $\Sigma$-structures and homomorphisms.  Positive-existential formulas are ones built from atomic formulas with an arbitrary use of $\bigwedge$, $\bigvee$ and $\exists$ (but no other connectives or quantifiers).  A \textsl{basic sentence} of $\LL_{\infty\kappa}$ is a formula of the form
\[  \forall\xx(\phi\implies\psi)  \]
with no free variables, where $\phi$ and $\psi$ are positive-existential.  A \textsl{basic theory} $T$ is one axiomatized by a set of basic sentences.

Let $T$ be a basic theory, and $\Mod(T)$ the category whose objects are models of $T$ and whose morphisms are homomorphisms of $\Sigma$-structures.  Let $\lambda\geq\kappa$ be a regular cardinal that exceeds the size $|I|$ of any conjunction $\bigwedge_{i\in I}\alpha_i$ that occurs in a sentence belonging to $T$.  Then the inclusion
\[  \Mod(T) \inc \str(\Sigma)  \]
creates $\lambda$-filtered colimits.  $\Mod(T)$ will in fact be an accessible category, a fortiori $\mu$-mono-generated for some $\mu\geq\lambda$.

The target categories of interest to us will be the categories of models of basic \textsl{universal} theories.

\begin{defn}
$T$ is a \textsl{basic universal theory} in $\LL_{\infty\kappa}$ if it is axiomatized by a set of sentences of the form
\[  \forall\xx(\phi\implies\psi)  \]
where $\phi$ and $\psi$ are built from atomic formulas with an arbitrary use of $\bigwedge$ and $\bigvee$ (but no quantifiers or other connectives).

Let $\lambda\geq\kappa$ be a regular cardinal.  $\forall_{\lambda\kappa}$ will denote the class of basic universal theories in $\LL_{\infty\kappa}$ where the use of conjunction $\bigwedge_{i\in I}$ is permitted only for $|I|<\lambda$.
\end{defn}

\begin{exam}
$\bullet$ In the above definitions, the empty conjunction is understood as the logical constant `True' $\top$ and the empty disjunction as the logical constant `False' $\bot$.  Thus universal Horn theories, axiomatized by sentences of the form
\[  \forall\xx\big(\bigwedge_{i\in I}\alpha_i \implies \beta  \big) \]
\[  \forall\xx\big(\bigwedge_{i\in I}\alpha_i \implies \bot  \big) \]
with atomic $\alpha_i$, $\beta$, are basic universal, and belong to $\forall_{\lambda\kappa}$ if $|I|<\lambda$ in all the axioms.  Specializing further, quasi-varieties of algebras, axiomatized by implications between non-empty conjunctions of terms in a language with no relation symbols, are basic universal.  See Rosick\'{y}~\cite{R} for an intrinsic categorical characterization of categories of the form $\Mod(T)$ where $T$ is a universal Horn theory resp.\ quasi-variety.

$\bullet$ The theory of torsion groups, axiomatized as groups together with
\[  \forall x\big(\bigvee_{n\in\NN} x^n = \mathbf{1}  \big)  \]
is $\forall_{\omega\omega}$.  Since this category does not have all products, it is not the category of models and homomorphisms of any universal Horn theory.

$\bullet$ If $T$ is a basic theory in $\LL_{\omega\omega}$, its Skolemization (as constructed below)\ is $\forall_{\omega\omega}$.  Indeed, let $\forall\xx(\phi\implies\psi)$ be a basic sentence, where $\phi$, $\psi$ have prenex form $\exists\uu\phi_0$, $\exists\vv\psi_0$.  Without loss of generality, $\uu$ does not occur in $\psi_0$ and $\vv$ does not occur in $\phi_0$.  Then
\begin{multline*}
\forall\xx(\phi\implies\psi) \biimp \forall\xx(\exists\uu\phi_0\implies\exists\vv\psi_0) \biimp \\
\biimp \forall\xx\forall\uu\exists\vv(\phi_0\implies\psi_0) \biimp \forall\xx\forall\uu(\phi_0\implies\psi_1)
\end{multline*}
where $\psi_1$ is the result of replacing $\vv$ by Skolem functions.  The prenex form of a positive-existential formula of $\LL_{\lambda\kappa}$ will belong to $\LL_{\lambda\mu}$ for $\mu\geq\kappa$; modulo that, the argument works for infinitary logics.  (Note that formulas may fail to have prenex forms within $\LL_{\lambda\kappa}$.)  However, Skolemization enlarges the language, and thus changes (restricts)\ the notion of morphism of models.
\end{exam}

Let $T$ be a basic universal theory in the language $\Sigma$ and $f:X\ra Y$ a morphism of $T$-models.  Let $U$ be the set-theoretic image of $f$ in the set(s)\ underlying $Y$.  Equip $U$ with the $\Sigma$-structure induced from $Y$.  Since $T$ is a theory axiomatized by universal sentences, $U$ will be a $T$-model too, and $X\epi U\into Y$ morphisms of $T$-models.  We will refer to this as the \textsl{image factorization} of $f$.

\begin{lemma}  \label{factor}
Let $T\in\forall_{\lambda\kappa}$ in the language $\Sigma$, let $\cA$ be a $\lambda$-mono generated category, and $F:\cA\ra\Mod(T)$ a functor that turns $\lambda$-directed colimits of monomorphisms into colimits (not necessarily of monomorphisms).  Let $\cS$ be a $\lambda$-dense set of $\lambda$-spans between objects $X,Y$ of $\cA$, and $X\overset{f}{\linto}U\overset{g}{\into}Y\in\cS$.  Let $F(X)\linto U_0\lepi F(U)$ and $F(U)\epi U_1\into F(Y)$ be the image factorizations of $F(f)$ resp.\ $F(g)$ in $\Mod(T)$.  Then $U_0$ and $U_1$ are canonically isomorphic.
\end{lemma}

\begin{proof}
To begin with, let $z_1,z_2$ be elements of $F(U)$ such that $F(f)(z_1)=F(f)(z_2)$ in $F(X)$.  We claim that then $F(g)(z_1)=F(g)(z_2)$ as elements of $F(Y)$.

Write $X$ as a $\lambda$-directed colimit of $\lambda$-generated subobjects $X_\alpha$ and monomorphisms; without loss of generality, the $X_\alpha$ contain $U$.  Then $F(X)$ is the $\lambda$-directed colimit of the $\{\,F(X_\alpha)\,\}$.  Since $\lambda$-directed colimits are computed on underlying structures in $\Mod(T)$, there is $X_\alpha$ with $i:U\ra X_\alpha$ such that $F(i)(z_1)=F(i)(z_2)$.  Apply density to $X\la U\ra Y$ and $X\la X_\alpha$ to find appropriate $X\la V\ra Y$, $U\ra V$ and $X_\alpha\ra V$.  Consider the diagram
\begin{equation*}  \tag{$\dag$}
\xymatrix{  & F(V) \ar[dl]\ar[dr] \\
F(X) & F(X_\alpha)\ar[l]\ar[u] & F(Y) \\
& F(U) \ar[u]\ar[ul]^{F(f)}\ar[ur]_{F(g)}\ar@/_9mm/[uu] }
\end{equation*}
Since the composites $U\ra X_\alpha\ra V\ra X$ and $U\ra V\ra X$ both equal $f:U\ra X$, and $V\ra X$ is mono, $U\ra X_\alpha \ra V=U\ra V$; it follows that the diagram commutes.  Thence $F(g)(z_1)=F(g)(z_2)$ as claimed.  A symmetric argument establishes that $F(g)(z_1)=F(g)(z_2)$ implies $F(f)(z_1)=F(f)(z_2)$.

Define a map $w$ from the set underlying $U_0$ to the set underlying $U_1$ as follows: for $u\in U_0$, find $z\in F(U)$ with $u=F(f)(z)$ and let $w(u)=F(g)(z)$.  By the above, $w$ is well-defined; it is immediate that it is bijective and preserves the interpretation of all function symbols.  To show that $w$ is an isomorphism of $U_0$ and $U_1$ as $\Sigma$-structures, it remains to check that it preserves the interpretation of relation symbols in $\Sigma$.  So let $R$ be a relation symbol of arity $\mu<\kappa$ and let $\uu=\{u_i\;|\;i<\mu\}$ be a tuple of elements of $U_0$ such that $U_0\models R(\uu)$.  Let $\zz\in U$ be such that $\uu=F(f)(\zz)$.  Similarly to the argument above, one can find a $\lambda$-generated subobject $X_\alpha$ of $X$, with an inclusion $i:U\ra X_\alpha$, such that $F(X_\alpha)\models R\big(i(\zz)\big)$.  Repeat the above argument to construct the diagram $(\dag)$.  Since $F(X_\alpha)\ra F(V)\ra F(Y)$ are homomorphisms of $\Sigma$-structures, $F(Y)\models R\big(F(g)(\zz)\big)$.  But this means $U_1\models R(w(\uu))$, since $U_1$ is induced from $F(Y)$ as $\Sigma$-substructure.  Conversely, $U_1\models R(w(\uu))$ implies $U_0\models R(\uu)$.
\end{proof}

Under the assumptions of the lemma, one can associate to any $X\linto U\into Y\in\cS$ a diagram in $\Mod(T)$
\[  \xymatrix{ F(X) & W\ar@{>->}[l]\ar@{>->}[r] & F(Y) \\
& F(U) \ar[lu]\ar[ru]\ar@{->>}[u] }  \]
where $F(X)\linto W\lepi F(U)$ and $F(U)\epi W\into F(Y)$ are (isomorphic to)\ image factorizations.  We will refer to $F(X)\linto W\into F(Y)$ as an image factorization of $F(X)\la F(U)\ra F(Y)$.  Though only defined up to isomorphism, note that if $F(X)\linto W_0\into F(Y)$ resp.\ $F(X)\linto W_1\into F(Y)$ are image factorizations of $F(X)\la F(U_0)\ra F(Y)$ resp.\ $F(X)\la F(U_1)\ra F(Y)$, then any morphism of spans $s:U_0\ra U_1$ in $\cS$ induces a (unique)\ morphism $w:W_0\ra W_1$ making
\[  \xymatrix{
& F(U_1)\ar@{>>}[d]\ar[ddl]\ar[ddr] \\
& W_1\ar@{>->}[dl]\ar@{>->}[dr] \\
F(X) && F(Y) \\
& W_0\ar[uu]_w\ar@{>->}[lu]\ar@{>->}[ru] \\
& F(U_0)\ar@{>>}[u]\ar[luu]\ar[ruu]\ar@/_30mm/[uuuu]_{F(s)}
} \]
commutative.

\begin{thm}  \label{eklof}
Let $\cA$ be a $\lambda$-mono-generated category, $T\in\forall_{\lambda\kappa}$, and $F:\cA\ra\Mod(T)$ a functor that takes $\lambda$-directed colimits of monos to colimits.  Then $X\sim_\lambda Y$ implies $F(X)\sim_\lambda F(Y)$.
\end{thm}

\begin{proof}
Let $\cS$ be a $\lambda$-dense set of $\lambda$-spans between $X$ and $Y$.  We claim that the set $\cS_T$ of image factorizations of $F(X)\la F(U)\ra F(Y)$, with $X\la U\ra Y\in\cS$, is $\lambda$-dense between $F(X)$ and $F(Y)$.  Indeed, $\cS_T$ is non-empty.  Let $F(X)\linto W\into F(Y)\in\cS_T$, a factorization of $F(X)\la F(U)\ra F(X)$, and $F(X)\linto G$ be given, with $\lambda$-generated $G$.  Since $\lambda$-filtered colimits in $\Mod(T)$ are created by the colimits of underlying structures, this implies that there are elements $\uu=\{u_i\;|\;i\in I\}$ in $G$, with $|I|<\lambda$, whose closure under function and constant symbols is the set underlying $G$.  Without loss of generality, $F(X)\linto G$ is injective on underlying sets.  (If not, let $F(X)\linto G_0\linto G$ be its image factorization, and solve the span extension problem for $F(X)\linto G_0$.)  We will identify $G$ with its image in $F(X)$.  Write $X$ as a $\lambda$-directed colimit of $\lambda$-generated subobjects $X_\alpha$.  There exists $i_\alpha:X_\alpha\into X$ such that the set-theoretical image of $F(i_\alpha)$ contains $\uu$.  Apply density of $\cS$ to $X\linto U\into Y$ and $X\linto X_\alpha$ to find appropriate $X\la V\ra Y\in\cS$, $s:U\ra V$ and $X_\alpha\ra V$.  Consider the diagram
\[  \xymatrix{
F(X_\alpha)\ar[r]\ar[dd] & F(V)\ar@{>>}[d]\ar[ddl]\ar[ddr] \\
& Z\ar@{>->}[dl]\ar@{>->}[dr] \\
F(X) && F(Y) \\
& W\ar[uu]_w\ar@{>->}[lu]\ar@{>->}[ru] \\
G\ar[uu]\ar@{.>}[ruuu] & F(U)\ar@{>>}[u]\ar[luu]\ar[ruu]\ar@/_30mm/[uuuu]_{F(s)}
} \]
where $w$ is induced by $F(s)$ as above.  The image of $Z$ in $F(X)$ contains $\uu$.  Recall that $Z$ is an induced substructure of $F(X)$, $G$ is a substructure of $F(X)$, and $\uu$ generates $G$.  Thence the dotted map $G\ra Z$, a priori only defined on $\uu$, extends to give a morphism $G\ra Z$, verifying the ``back'' case of the $\lambda$-density of $\cS_T$.  The ``forth'' case is symmetric.
\end{proof}

\begin{exam}
Let $F:\Grp\ra\Ab$ be the abelianization functor, sending a group $G$ to $G/[G,G]$.  $F$ preserves filtered colimits and Thm.~\ref{eklof} applies; so, if two groups are $\lambda$-equivalent, so are their abelianizations.  Here, $\lambda$-equivalence coincides with the usual notion of $\LL_{\infty\lambda}$-equivalence in the language of groups; see Thm.~\ref{karp}.  That abelianization preserves $\LL_{\infty\lambda}$-equivalence can be proved directly, using that $G/[G,G]$ is a quotient construction by an $\LL_{\omega_1\omega}$-definable equivalence relation, though the details are somewhat unpleasant.  To show that one can associate to a set of partial isomorphisms between two groups a set of partial isomorphisms between their abelianizations, it seems that one needs to go through the equivalent of Lemma~\ref{factor}, since the abelianization of a partial isomorphism need not be one.
\end{exam}

\begin{cor}  \label{pres2}
Let $\cA$ be a $\lambda$-mono-generated category with the property that directed colimits of monomorphisms exist and are mono.  Suppose the functor $F:\cA\ra\B$ preserves monomorphisms and $\lambda$-directed colimits of monomorphisms \textsl{or} $F$ takes $\lambda$-directed colimits of monos to colimits, and $\B$ is equivalent to $\Mod(T)$ for a $\forall_{\lambda\kappa}$ theory $T$.  Then for all regular $\mu\geq\lambda$, if $X\sim_\mu Y$ then $F(X)\sim_\mu F(Y)$.
\end{cor}

Indeed, every $\lambda$-mono-generated category $\cA$ with the above property is $\mu$-mono-generated for every regular $\mu\geq\lambda$; the proof is analogous to that of Prop.~4.1 of Beke-Rosick\'{y}~\cite{BR}.  Now apply Prop.~\ref{pres} resp.\ Thm.~\ref{eklof}. \qed

Two structures --- say, two algebraically closed fields --- can be considered as objects of many $\lambda$-mono-generated categories: of algebraically closed fields, fields, rings, abelian groups, $(+,\times)$-structures$\dots$ each carrying its notion of $\lambda$-equivalence.  We will now give some sufficient conditions for a functor $F:\cA\ra\B$ to preserve and \textsl{reflect} $\lambda$-equivalence.

\begin{prop}  \label{reflect2}
(a) Suppose $\cA$ and $\B$ are $\lambda$-mono-generated categories, $F:\cA\ra\B$ is a functor, and $G:\B\ra\cA$ is a functor preserving $\lambda$-equivalence with the property that $GF(X)$ is isomorphic to $X$ for all objects $X\in\cA$.  Then $F(X)\sim_\lambda F(Y)$ implies $X\sim_\lambda Y$. (b) Let $T$ be a basic universal theory with $\Mod(T)$ $\lambda$-mono-generated and let the adjunction $F:\Mod(T)\leri\B:G$ exhibit $\Mod(T)$ as a reflective subcategory of a $\lambda$-mono-generated category $\B$.  If the inclusion $F$ preserves $\lambda$-filtered colimits then $X\sim_\lambda Y$ if and only if $F(X)\sim_\lambda F(Y)$.
\end{prop}

\begin{proof}
(a) If $F(X)\sim_\lambda F(Y)$ then $X\sim_\lambda GF(X)\sim_\lambda GF(Y)\sim_\lambda Y$ by Cor.~\ref{mono-eq}. (b) The left adjoint $G$ preserves all colimits and the right adjoint $G$ preserves monos.  Apply Thm.~\ref{eklof}, (a), and Prop.~\ref{pres}.
\end{proof}

Note that the condition on $F$ is satisfied for any accessible adjunction.  Prop.~\ref{reflect2} applies, in particular, to the inclusion of a category of sheaves of algebras on a Grothendieck site in the corresponding category of presheaves, localizations of Grothendieck abelian categories, and reflections of the form $\Mod(T^+)\leri\Mod(T)$ where $T$ and $T^+$ are both quasi-varieties of algebras, and the theory $T^+$ extends $T$ in the same language.

One can also tweak the assumptions of Prop.~\ref{pres} to conclude that a functor preserves and reflects $\lambda$-equivalence.

\begin{prop}  \label{reflect}
Let $\cA$ and $\B$ be $\lambda$-mono-generated categories and $F:\cA\ra\B$ a functor preserving monomorphisms and $\lambda$-directed colimits of monomorphisms.  Suppose also
\begin{itemize}
\item[(i)] $F$ is (equivalent to)\ the inclusion of a full subcategory
\item[(ii)] if $X$ is $\lambda$-generated, so is $F(X)$.
\end{itemize}
Then $X\sim_\lambda Y$ if and only if $F(X)\sim_\lambda F(Y)$.
\end{prop}

\begin{proof}
Only the `if' direction is new.  To avoid clutter, we suppress $F$ from the notation, but make clear whether we work in $\B$ or its full subcategory $\cA$.  Suppose $X,Y\in\cA$ are such that $X\sim_\lambda Y$ in $\B$, thanks to the $\lambda$-dense set of $\lambda$-spans $\cS_\B$.  Let $\cS_\cA$ be the set of spans in $\cA$ defined as follows: $X\linto A\into Y\in\cS_\cA$ if and only if there exists $X\linto U\into Y\in\cS_\B$ and a map of spans
\begin{equation}  \tag{$\ddag$}
\xymatrix{ X & U\ar[l]\ar[r] & Y \\
& A \ar[lu]\ar[ru]\ar[u] }
\end{equation}
in $\B$.  We claim that $\cS_\cA$ is $\lambda$-dense between $X$ and $Y$ in $\cA$.  $\cS_\cA$ is non-empty: write $X$ as a $\lambda$-directed colimit of $\lambda$-generated subobjects $X_i$, $i\in I$, in $\cA$.  Pick any $j\in I$; then, both in $\cA$ and $\B$, $X\linto X_j$ is mono with $\lambda$-generated $X_j$.  Apply density to $X\linto X_j$ and an arbitrary span $X\linto W\into Y\in\cS_\B$ to deduce the existence of a diagram
\[ \xymatrix{ & U \ar[ld]\ar[rd] \\
X & X_j\ar[l]\ar[u] & Y \\
& W \ar[lu]\ar[ru]\ar@/_5mm/[uu] } \]
with $X\linto U\into Y\in\cS_\B$.  Thence $X\linto X_j\into U\into Y\in\cS_\cA$.

Let now $X\linto A\into Y\in\cS_\cA$ and $X\linto G$ be given, where $G$ is $\lambda$-generated in $\cA$ (hence, in $\B$).  Let $X\linto U\into Y\in\cS_\B$ be so that ($\ddag$)\ holds.  Writing $X$ as a $\lambda$-directed colimit of $\lambda$-generated subobjects $X_i$, $i\in I$ again, since both $G$ and $U$ are $\lambda$-generated, there exists $j\in I$ such that both $X\linto G$ and $X\linto U$ factor through $X_j$.  Apply density to $X\linto X_j$ and $X\linto U\into Y$ to deduce the existence of $X\linto V\into Y\in\cS_\B$ and morphisms
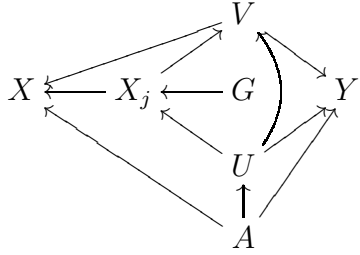
\[ \xymatrix{
&& V \ar[dll]\ar[dr] \\
X & X_j \ar[l]\ar[ur] & G \ar[l] & Y \\
&& U \ar[ul]\ar[ur]\ar@/_5mm/[uu] \\
&& A \ar[uull]\ar[u]\ar[uur] } \]
Note that $X\linto X_j\into V\into Y\in\cS_\cA$.  Both $U\ra X_j\ra V\ra X$ and $U\ra V\ra X$ equal $U\ra X$; since $V\ra X$ is mono, $U\ra X_j\ra V$ equals $U\ra V$.  It follows that the diagram commutes.  Thence $A\ra U\ra X_j$ is a morphism of spans from $X\linto A\into Y$ to $X\linto X_j\into Y$ verifying the ``back'' part of the density of $\cS_\cA$.  The other part is symmetric.
\end{proof}

Although the assumptions of Prop.~\ref{reflect} seem restrictive, they cannot be weakened in an obvious way.  When $\cA$ and $\B$ are $\lambda$-mono-accessible, the Uniformization Theorem of Makkai and Par\'{e}, cf.\ \cite{MP} 2.4.9, implies that their degree of accessibility can be raised so that condition $(ii)$ is satisfied.  As to condition $(i)$, consider the functor $F:\Set\times\Set\ra\Set$ that is projection on the first coordinate.  It is full (but not faithful), preserves monos and $\lambda$-directed colimits of monos, and $\lambda$-generated objects, for all regular $\lambda$; but it does not reflect $\sim_\lambda$ for any $\lambda$.  On the other hand, consider the underlying set functor $F:Grp\ra\Set$ on the category of groups.  It is faithful (but not full), preserves monos and $\lambda$-directed colimits of monos for all regular $\lambda$, and $\lambda$-generated objects for each $\lambda\geq\aleph_1$; but it does not reflect $\sim_\lambda$ for any $\lambda$.

\begin{exam}
Recall that any locally presentable category $\K$ is equivalent to the category of models of an essentially algebraic theory $T$ in some signature $\Sigma$, cf.\ \cite{AR} 3.36.  Consider the adjunction
\[  F: \K \leftrightarrows \str(\Sigma): G \]
where $F$ is the inclusion and $G$ its left adjoint.  The assumptions of Prop.~\ref{reflect} are satisfied for any $\lambda$ such that $\K$ is locally $\lambda$-presentable.  (Assumption (ii)\ holds since in an essentially algebraic signature $\Sigma$, a $T$-model is $\lambda$-generated if and only if it is the closure of less than $\lambda$ of its elements under the function symbols and constants in $\Sigma$.  This could fail for a finite limit theory containing relation symbols.)  Hence, for $T$-models $X$, $Y$, one has $X\sim_\lambda Y$ in $\K$ if and only if $X\sim_\lambda Y$ in $\str(\Sigma)$. 
\end{exam}

\begin{exam}
Let $T$ be a basic theory in the signature $\Sigma$ and consider the inclusion $F:\Mod(T)\ra\str(\Sigma)$.  Again, by the Uniformization Theorem, there exists a regular cardinal $\lambda$ such that $\Mod(T)$ and $\str(\Sigma)$ are $\lambda$-mono-generated, and $F$ preserves $\lambda$-generated objects.  If $F$ preserves monos then Prop.~\ref{reflect} applies.

Note that one can always achieve that the inclusion $F$ preserve monos without changing the class of $T$-models by passing to an extension of $T$ by definition.  Add the binary predicate $n(-,-)$ (to be thought of as `not equal')\ to $\Sigma$ and let $T^+$ be $T$ together with the basic universal axioms
\[   \big( n(x,y) \wedge x=y \big) \ra \bot \]
\[   n(x,y) \,\vee\, x=y \]
Then $T$-models can be identified with $T^+$-models, but morphisms in $\Mod(T^+)$ preserve $n(-,-)$, i.e.\ are injective on underlying structures, so $\Mod(T^+)\ra\str(\Sigma)$ preserves monos.
\end{exam}

We now turn to the relation between $\lambda$-equivalence in a category of structures and the classical syntactic notion of $\LL_{\infty\lambda}$-equivalence.  Let $\lambda$ be a regular cardinal and $\Sigma$ a $\lambda$-ary signature.  We can associate two categories $\emb(\Sigma)\inc\str(\Sigma)$ to this signature.  The objects of each are the $\Sigma$-structures, but morphisms of $\str(\Sigma)$ are homomorphisms of $\Sigma$-structures, while morphisms of $\emb(\Sigma)$ are embeddings of substructures.  Both of these categories are $\lambda$-mono-generated; in fact, $\str(\Sigma)$ is locally presentable and $\emb(\Sigma)$ is accessible.  An object of $\str(\Sigma)$ is $\lambda$-generated if and only if it is generated by less than $\lambda$ of its elements and the sum of the cardinalities of the interpretations of its relation symbols is less than $\lambda$, cf.\ \cite{AR}~5.B.  An object of $\emb(\Sigma)$ is $\lambda$-generated if and only if it is generated by less than $\lambda$ of its elements in the classical model-theoretic sense (that is, the closure of those elements and the constants under the function symbols is all of the domain).  The main result is that for $\Sigma$-structures $X$ and $Y$, all three notions -- $\LL_{\infty\lambda}$-equivalence and $\sim_\lambda$ as objects of $\str(\Sigma)$ resp.\ of $\emb(\Sigma)$ --  coincide.

\begin{thm}  \label{karp}
Let $\lambda$ be a regular cardinal and $\Sigma$ be a $\lambda$-ary signature.  Then for the $\Sigma$-structures $X$ and $Y$, the following conditions are equivalent:
\begin{enumerate}
\item[(1)] $X$ and $Y$ are $\LL_{\infty\lambda}$-elementary equivalent
\item[(2)] there is a non-empty set $I$ of partial isomorphisms between $X$ and $Y$ satisfying the $<\lambda$-back-and-forth property
\item[(3)] $X\sim_\lambda Y$ in $\emb(\Sigma)$
\item[(4)] $X\sim_\lambda Y$ in $\str(\Sigma)$.
\end{enumerate}
\end{thm}

\begin{proof}
The equivalence of (1)\ and (2)\ is Karp's theorem.  We use the formulation of Nadel--Stavi~\cite{NS}~Theorem 1.1; see e.g.\ Dickmann~\cite{D}~Corollary~5.3.22 for a complete discussion.

As to the relation of (2)\ and (3), recall that $I$ is a set of isomorphisms $f:U\to U'$ such that $U$ is a substructure of $X$ and $U'$ is a substructure of $Y$.  The category of partial isomorphisms between $X$ and $Y$, with morphisms extensions of partial isomorphisms, is equivalent to the category of spans between $X$ and $Y$ in $\emb(\Sigma)$.  We will identify the two notions and only need to check implications between the respective variants of the back-and-forth property (describing just the ``back'' direction, by symmetry).

(2)$\dra$(3): Given a span $X\linto U\into Y$ and $\lambda$-generated mono $X\linto G$ in $\emb(\Sigma)$, select a set $Z$ of $<\lambda$ generators of $G$ and use the back part of (2)\ to find a span $X\linto V\into Y$ that $X\linc U\inc Y$ maps to in $\emb(\Sigma)$, with compatible map on underlying sets $Z\ra |V|$.  Since $X\linto V$ is the embedding of a substructure, $Z\ra |V|$ extends to a compatible morphism $G\into V$ in $\emb(\Sigma)$.

(3)$\dra$(2): given a partial isomorphism $X\linc U\inc Y$ and subset $Z$ of the set underlying $X$, with $\card(Z)<\lambda$, let $G$ be the $\Sigma$-substructure of $X$ generated by $Z$.  Thus $G$ is a $\lambda$-generated object of $\emb(\Sigma)$, with $X\linto G$.  The back property in (3)\ yields a partial isomorphism $X\linc V\inc Y$ with compatible $G\ra V$, hence the image of $V$ in $X$ containing $Z$ as desired.

(3)$\dra$(4): apply Prop.~\ref{pres} to the inclusion $\emb(\Sigma)\inc\str(\Sigma)$.

(4)$\dra$(3): let $\cS$ be a $\lambda$-dense set of $\lambda$-spans between $X$ and $Y$ in $\str(\Sigma)$.  The empty theory in the signature $\Sigma$, whose category of models is $\str(\Sigma)$, is basic universal.  The proof of Thm.~\ref{eklof}, applied to the identity functor $\str(\Sigma)\ra\str(\Sigma)$, shows that the set $\cS_\emptyset$ of image factorizations of elements of $\cS$, is $\lambda$-dense in $\str(\Sigma)$.  Note that the image factorization of a span in $\str(\Sigma)$ is a span in $\emb(\Sigma)$.  We only have to show that $\cS_\emptyset$ is $\lambda$-dense in $\emb(\Sigma)$ as well.  Indeed, let $X\linto U\into Y\in\cS_\emptyset$ and $X\linto G$ be given, where $G$ is $\lambda$-generated in $\emb(\Sigma)$.  Let now $G_0$ be identical to the structure $G$ with the exception that all relations in $\Sigma$ are interpreted by the empty set.  Then $G_0$ is $\lambda$-generated in $\str(\Sigma)$.  Apply density to $X\linto U\into Y$ and $X\linto G_0$ to deduce the existence of $X\linto V\into Y\in\cS_\emptyset$ with compatible maps $G_0\ra V$ and $U\ra V$ in $\str(\Sigma)$.  Since $X\linto U$ and $X\linto V$ are embeddings, so is $U\ra V$.  Endowing $G_0$ with the interpretation of the relation symbols induced from $X$, one regains $G$ and a commutative diagram that verifies the ``back'' condition in $\emb(\Sigma)$; and analogously for the ``forth''.  This completes the proof of Theorem~\ref{karp}.
\end{proof}

\section{Elementary embeddings}
Parallel to the theory of $\lambda$-equivalence, there exists a theory of $\lambda$-embeddings in categories.

\begin{defn}  \label{elem}
A morphism $f:X\to Y$ in a category $\cA$ is called a $\lambda$-\textsl{embedding} if there is a $\lambda$-dense set $\cS$ of spans between $X$ and $Y$ such that for any monomorphism $g:G\into X$ with $\lambda$-generated $G$ there exist $X\llla{u}U\llra{v}Y\in\cS$ and $t:G\ra U$ such that in the diagram
\[ \xymatrix{ X\ar[rr]^f && Y \\
& U \ar[lu]_u\ar[ru]^v \\
& G\ar[luu]^g\ar[u]_t } \]
$ut=g$ and $vt=fg$.
\end{defn}

\begin{rem}\label{re3.2}
\begin{itemize}
\item[(1)] It would be equivalent to demand that there exist a $\lambda$-dense set $\cS$ of $\lambda$-spans with the above property.  The proof is similar to that of Cor.~\ref{lambda}.  Moreover, if there is a $\lambda$-dense set of spans (resp.\ $\lambda$-dense set of $\lambda$-spans) satisfying the definition, then, without loss of generality, we can take it to be the greatest $\lambda$-dense set of spans (resp.\ greatest $\lambda$-dense set of $\lambda$-spans), cf.\ Remark~\ref{max}(3).
\item[(2)] If $\lambda<\kappa$ then any $\kappa$-embedding is also a $\lambda$-embedding.
\item[(3)] Any isomorphism $f:X\to Y$ is a $\lambda$-embedding.  Just let $\cS$ consist of the span $X\llla{\id}X\llra{f}Y$, cf.\ Remark~\ref{max}(1).
\end{itemize}
\end{rem}

\begin{rem}
If the set of all $\lambda$-spans between $X$ and $Y$ is $\lambda$-dense then every monomorphism $f:X\into Y$ is a $\lambda$-embedding.  Indeed, given $g:G\into X$, set $U=G$, $t=\id_G$, $u=g$ and $v=fg$.

The converse is not true, since the fact that all monos from $X$ to $Y$ are $\lambda$-embeddings can also hold vacuously.  Work, for example, in the category $\Set\times\Set$.  Consider cardinals $\lambda<\mu_1<\mu_2$ and let $X=(\mu_1,\mu_2)$ and $Y=(\mu_2,\mu_1)$.  Then the set of all $\lambda$-spans between $X$ and $Y$ is $\lambda$-dense (this is just Example~\ref{set} applied coordinatewise)\ but there exists no monomorphism from $X$ to $Y$, or from $Y$ to $X$.  This example also shows that while the existence of a $\lambda$-embedding between two objects obviously implies that they are $\lambda$-equivalent, the converse is not true.
\end{rem}

The next several propositions show that $\lambda$-embeddings satisfy the expected properties.

\begin{lemma}   \label{embmono}
Let $\cA$ be a $\lambda$-mono-generated category. Then any $\lambda$-embedding is a monomorphism.
\end{lemma}

\begin{proof}
Let $f:X\to Y$ be a $\lambda$-embedding.  Since $\cA$ is $\lambda$-mono-generated, it suffices to show that for any two monomorphisms $x,y:G\into X$ with $\lambda$-generated $G$, $fx=fy$ implies that $x=y$.  Writing $X$ as a $\lambda$-directed colimit of $\lambda$-generated subobjects, we see that $x$ and $y$ factorize through $g:G'\into X$ for some $\lambda$-generated $G'$, i.e.\ there exist $x',y':G\to G'$ with $x=gx'$ and $y=gy'$. Consider
\[ \xymatrix{ X\ar[rr]^f && Y \\
& U \ar[lu]_u\ar[ru]^v \\
& G'\ar[luu]^g\ar[u]_t } \]
with $ut=g$ and $vt=fg$ as in Def.~\ref{elem}.  Then
\[   vtx'=fgx'=fx=fy=fgy'=vty'  \]
and, since $g$, $t$, and finally $vt$ is a monomorphism, $x'=y'$.  Thus $x=y$.
\end{proof}

This can be improved.  Recall that a monomorphism $f:X\to Y$ is $\lambda$-\textit{pure in $\cA_\mono$} if, given any commutative square
$$
\xymatrix@=2pc{
X \ar[r]^{f} & Y \\
A \ar [u]^u \ar [r]_{g} &
B \ar[u]_v
}
$$
in $\cA_\mono$ with $A$ and $B$ $\lambda$-generated, there is a mono $t:B\to X$ such that $tg=u$.

\begin{prop}
Let $\cA$ be a $\lambda$-mono-generated category. Then any $\lambda$-embedding is $\lambda$-pure in $\cA_\mono$.
\end{prop}
\begin{proof}
Let $f:X\to Y$ be a $\lambda$-embedding and consider a commutative square $fu=vg$
\[ \xymatrix{
& U_1\ar[dl]_{u_1}\ar[dr]^{v_1} \\
X \ar[rr]^(.3){f} && Y \\
A \ar[u]^u\ar[r]_{t_0} \ar@/_16pt/[rr]_g & U_0 \ar[lu]^{u_0}\ar[ru]_{v_0}\ar@{.>}[uu]_(.3)w & B\ar[u]_{v}\ar@/_45pt/[uul]_{t_1} } \]
where $A$ and $B$ are $\lambda$-generated.  Thus there exist a span $X\llla{u_0}U_0\llra{v_0}Y$ and a morphism $t_0:A\to U_0$ such that
$u_0t_0=u$ and $v_0t_0=fu=vg$.  There exist a span $X\llla{u_1}U_1\llra{v_1}Y$ and morphisms $w:U_0\to U_1$, $t_1:B\to U_1$ such that
$u_1w=u_0$, $v_1w=v_0$ and $v=v_1t_1$. Hence
\[  v_1t_1g=vg=fu=v_0t_0=v_1wt_0  \]
and, since $v_1$ is a monomorphism, $wt_0=t_1g$. Thus
\[ u=u_0t_0=u_1wt_0=u_1t_1g \, .  \]
Therefore $f$ is $\lambda$-pure.
\end{proof}

\begin{lemma}  \label{embediso}
Let $\cA$ be $\lambda$-mono-generated and let $f:X\to Y$ be a $\lambda$-embedding with $X$ and $Y$ $\lambda$-generated.  Then $f$ is an isomorphism.
\end{lemma}

\begin{proof}
Let $\cS$ be a $\lambda$-dense set of spans between $X$ and $Y$ to which Def.~\ref{elem} applies.  Since $X$ is $\lambda$-generated and $f$ is a $\lambda$-embedding, there exists $X\llla{u_0}U_0\llra{v_0}Y\in\cS$ and $t:X\to U_0$ such that $u_0t=\id_X$ and $f=v_0t$.  Hence $u_0$ is an isomorphism and therefore $t$ is an isomorphism.
\[ \xymatrix{ & U_1 \ar[dl]_{u_1}\ar[dr]^{v_1} & Y \ar[d]^{\id_Y}\ar[l]_s \\
X\ar[rr]^(.3)f && Y \\
X \ar[u]^{\id_X}\ar[r]_t & U_0 \ar[lu]^{u_0}\ar[ru]_{v_0}\ar@{.>}[uu]_(.3)w } \]
By the ``forth'' property applied to $\id_Y$ and $X\llla{u_0}U_0\llra{v_0}Y$, there exists $X\llla{u_1}U_1\llra{v_1}Y\in\cS$, $w:U_0\to U_1$ and $s:Y\to U_1$ such that $u_1w=u_0$, $v_1w=v_0$ and $v_1s=\id_Y$.  Hence $v_1$ and $u_1$ are isomorphisms and therefore $w$ is an isomorphism.  Since $f=v_0t=v_1wt$, $f$ is an isomorphism.
\end{proof}

\begin{cor}
Let $f:X\to Y$ be a morphism in a mono-generated category.  $f$ is an isomorphism if and only if there are arbitrarily large regular $\lambda$ (equivalently, for all regular $\lambda$)\ $f$ is a $\lambda$-embedding.
\end{cor}

This follows from Lemma~\ref{embediso} by an argument similar to that of Cor.~\ref{isochar}.

\begin{prop}  \label{compo}
In any $\lambda$-mono-generated category $\cA$, $\lambda$-embeddings are closed under composition.
\end{prop}

\begin{proof}
Let $\cS_{XY}$ resp.\ $\cS_{XZ}$ be $\lambda$-dense sets of $\lambda$-spans verifying that $f:X\to Y$ resp.\ $g:Y\to Z$ are $\lambda$-embeddings.  We claim that $\cS_{XY}\star\cS_{YZ}$, cf.\ Prop.~\ref{compose}, shows that $gf$ is a $\lambda$-embedding.  Consider a mono $x:G\to X$ with $\lambda$-generated $G$.  There is an $X\llla{s_1}S\llra{s_2}Y\in\cS_{XY}$ and $s:G\to S$ such that $s_1s=x$ and $s_2s=fx$.
\[ \xymatrix{ X \ar[rr]^f && Y \ar[rr]^g && Z \\
& S\ar[ul]_{s_1}\ar[ur]^{s_2} && T\ar[ul]_{t_1}\ar[ur]^{t_2}  \\
G \ar[uu]^x\ar[ur]^s\ar@/_3pt/[rrru]_t } \]
There is a $Y\llla{t_1}T\llra{t_2}Z\in\cS_{YZ}$ and $t:G\to T$ with $t_1t=s_2s$ and $t_2t=gs_2s$.  Note that $X\llla{x}G\llra{t_2t}Z\in\cS_{XY}\star\cS_{YZ}$ and the tautologous diagram with $\id_Gg=g$ and $gfx=gs_2s=t_2t$ shows that $gf$ is a $\lambda$-embedding.
\end{proof}

\begin{prop}
Let $\cA$ be a $\lambda$-mono-generated category and $g:Y\ra Z$ and $gf:X\ra Z$ be $\lambda$-embeddings.  Then $f:X\ra Y$ is a $\lambda$-embedding.
\end{prop}

\begin{proof}
Let $\cS_{XZ}$ resp.\ $\cS_{YZ}$ be $\lambda$-dense sets of $\lambda$-spans verifying that $gf$ resp.\ $g$ are $\lambda$-embeddings.  We claim that the composite $\cS_{XZ}\star\cS_{ZY}$ verifies that $f$ is a $\lambda$-embedding.  (Here $\cS_{ZY}=\cS_{YZ}$.)  By Prop.~\ref{compose}, $\cS_{XZ}\star\cS_{ZY}$ is $\lambda$-dense between $X$ and $Y$.  Consider a mono $x:G\to X$ with $\lambda$-generated $X$.  There exists $X\llla{t_1}T\llra{t_2}Z\in\cS_{XZ}$ and $t:G\to T$ so that $x=t_1t$ and $gfx=t_2t$.  Since $T$ is $\lambda$-generated, applying the assumption that $g$ is a $\lambda$-embedding to $ft_1:T\ra Y$, there exist $Y\llla{s_1}S\llra{s_2}Z\in\cS_{YZ}$ and $s:T\to S$ such that $ft_1=s_1s$ and $gft_1=s_2s$. The inner diamond of the diagram
\[  \xymatrix{ X \ar[rr]_{gf}\ar@/^15pt/[rrrr]^f && Z && Y \ar[ll]^g \\
& T\ar[ul]^{t_1}\ar[ur]_{t_2}\ar[rr]_s && S\ar[ul]^{s_2}\ar[ur]_{s_1} \\
&& G \ar[ul]^{t}\ar[ur]_{st} }  \]
commutes because
\[  s_2st = gft_1t = gfx = t_2t \, .  \]
Thus $X\llla{t_1t}G\llra{s_1st}Y$ belongs to $\cS_{XY}$. Since $x=t_1t\id_G$ and $fx=ft_1t=s_1st$, we get that $f$ is a $\lambda$-embedding.
\end{proof}

The next three statements concern the behavior of $\lambda$-embeddings under functors.  Both the assertions and their proofs parallel the case of $\lambda$-equivalences, with an extra step needed to verify the embedding condition.

\begin{thm}  \label{emb-pres}
(a) Let $\cA$ be a $\lambda$-mono-generated category and $F:\cA\ra\B$ a functor preserving monomorphisms and $\lambda$-directed colimits of monomorphisms.  If the morphism $f\in\cA$ is a $\lambda$-embedding, so is $F(f)$.  (b) Let $\cA$ be a $\lambda$-mono-generated category, $T\in\forall_{\lambda\kappa}$, and $F:\cA\ra\Mod(T)$ a functor that takes $\lambda$-directed colimits of monomorphisms to colimits.  If the morphism $f\in\cA$ is a $\lambda$-embedding, so is $F(f)$.
\end{thm}

\begin{proof}
(a) Let $X\la U_i\ra Y$, $i\in I$, be a $\lambda$-dense set of spans verifying that $f:X\ra Y$ is a $\lambda$-embedding.  We claim that the set of spans $F(X)\la F(U_i)\ra F(Y)$, $i\in I$, shows $F(f)$ to be a $\lambda$-embedding.  Indeed, by Prop.~\ref{pres}, this set of spans is $\lambda$-dense between $F(X)$ and $F(Y)$.  Now let a mono $G\to F(X)$ with $\lambda$-generated $G$ be given.  Write $X$ as a colimit of a $\lambda$-directed diagram of $\lambda$-generated objects $X_\alpha$ and monomorphisms; thence $F(X)$ is the $\lambda$-directed colimit of the $F(X_\alpha)$ along monomorphisms.  Find an element $x_\alpha:X_\alpha\ra X$ of the cocone so that $G\ra F(X)$ factors as $G\ra F(X_\alpha)\llra{F(x_\alpha)}F(X)$.  Since $f$ is a $\lambda$-embedding, there exist $X\llla{u}U_i\llra{v}Y$ and $t:X_\alpha\to U_i$ such that $ut=x_\alpha$ and $vt=fx_\alpha$.  The $F$-image of this data
\[ \xymatrix{  F(X) \ar[rr]^{F(f)} && F(Y) \\
& F(U_i) \ar[ur]^(.4){F(v)}\ar[ul]_(.4){F(u)} \\
G \ar[uu]\ar[r] & F(X_\alpha) \ar[uul]^(.4){F(x_\alpha)}\ar[u]_{F(t)} } \]
verifies that $F(f)$ is a $\lambda$-embedding.  (b) The argument is identical, with Thm.~\ref{eklof} in place of Prop.~\ref{pres}, and the image factorizations of the $F(X)\la F(U_i)\ra F(Y)$ verifying that $F(f)$ is a $\lambda$-embedding.  The final diagram is modified to
\[ \xymatrix{  F(X) \ar[rr]^{F(f)} && F(Y) \\
& W\ar@{>->}[ul]\ar@{>->}[ur] \\
& F(U_i) \ar[uur]\ar[uul]\ar@{->>}[u] \\
G \ar[uuu]\ar[r] & F(X_\alpha) \ar[uuul]^(.4){F(x_\alpha)}\ar[u]_{F(t)} } \]
\end{proof}

The proof of the following corollary is analogous to that of Cor.~\ref{pres2}.

\begin{cor}
Let $\cA$ be a $\lambda$-mono-generated category with the property that directed colimits of monomorphisms exist and are mono.  Suppose the functor $F:\cA\ra\B$ preserves monomorphisms and $\lambda$-directed colimits of monomorphisms \textsl{or} $F$ takes $\lambda$-directed colimits of monos to colimits, and $\B$ is equivalent to $\Mod(T)$ for a $\forall_{\lambda\kappa}$ theory $T$.  Then for all regular $\mu\geq\lambda$, if $f\in\cA$ is a $\mu$-embedding then so is $F(f)$.
\end{cor}

The analogue of Prop.~\ref{reflect2}(a) obviously holds for $\lambda$-embeddings in the case of a \textsl{natural} isomorphism between $\id_\cA$ and $GF$ (all notation and assumptions being as in Prop.~\ref{reflect2}), and the analogue of Prop.~\ref{reflect2}(b)\ holds as well.  So does the analogue of Prop.~\ref{reflect}:

\begin{prop}
Let $\cA$ and $\B$ be $\lambda$-mono-generated categories and $F:\cA\ra\B$ a functor preserving monomorphisms and $\lambda$-directed colimits of monomorphisms.  Suppose also
\begin{itemize}
\item[(i)] $F$ is (equivalent to)\ the inclusion of a full subcategory
\item[(ii)] if $X$ is $\lambda$-generated, so is $F(X)$.
\end{itemize}
Then $f:X\ra Y\in\cA$ is a $\lambda$-embedding if and only if $F(f)$ is.
\end{prop}

\begin{proof}
Only the `if' direction is new.  We follow the conventions of the proof of Prop.~\ref{reflect}; let $\cS_\B$ be a $\lambda$-dense set of $\lambda$-spans between $X$ and $Y$ in $\B$ verifying that $F(f)$ is a $\lambda$-embedding, and retain the definition of $\cS_\cA$.  By the proof of Prop.~\ref{reflect}, $\cS_\cA$ is a $\lambda$-dense set of $\lambda$-spans between $X$ and $Y$ in $\cA$.  To verify the embedding condition, let a mono $g:G\ra X$, with $\lambda$-generated $G$, be given in $\cA$.  Since $F(f)$ is a $\lambda$-embedding, there exist $X\llla{u}U\llra{v}Y\in\cS_\B$ and $t:G\ra U$ such that $g=ut$ and $fg=vt$. Write $X$ as a $\lambda$-directed colimit of $\lambda$-generated subobjects $X_i$, $i\in I$, in $\cA$ (hence, retaining these properties in $\B$).  Let $i\in I$ be such that $u:U\ra X$ factors through $X_i\ra X$.  Apply density to $X_i\ra X$ and $X\llla{u}U\llra{v}Y$ to find a span $X\linto W\into Y\in\cS_\B$ with suitable maps $U\ra W$ and $X_i\ra W$.  In the diagram
\[ \xymatrix{
X \ar[rr]^f && Y \\
X_i\ar[u]\ar[r] & W\ar[ul]\ar[ur] \\
G\ar[r]_t\ar@/^5mm/[uu]^g & U\ar[ul]\ar[u]\ar[uur]_v } \]
the composites $U\ra W\ra X$ and $U\ra X_i\ra W\ra X$ both equal $u:U\ra X$.  Since $W\ra X$ is mono, $U\ra W$ equals $U\ra X_i\ra W$.  This implies, by diagram chase, that $G\ra U\ra X_i\ra W\ra Y$ equals $fg$.  Since $G\ra U\ra X_i\ra X$ equals $g$, $X\la X_i\ra W\ra Y\in\cS_\cA$ and $G\ra U\ra X_i$ solve the embedding problem as desired.
\end{proof}

\begin{thm}
Let $\lambda$ be a regular cardinal and $\Sigma$ be a $\lambda$-ary signature.  Then for an embedding $f:X\ra Y$ of $\Sigma$-structures the following are equivalent:
\begin{enumerate}
\item[(1)] $f$ is an $\LL_{\infty\lambda}$-elementary embedding
\item[(2)] there is a non-empty set $I$ of partial isomorphisms between $X$ and $Y$ satisfying the $<\lambda$-back-and-forth property, and such that for every subset $Z$ of $X$ of cardinality less than $\lambda$ there is $h\in I$ such that $f(z)=h(z)$ for every $z\in Z$
\item[(3)] $f$ is a $\lambda$-embedding in $\emb(\Sigma)$
\item[(4)] $f$ is a $\lambda$-embedding in $\str(\Sigma)$.
\end{enumerate}
\end{thm}

\begin{proof}
The equivalence of (1)\ and (2)\ is in Dickmann~\cite{D}; see the observation after 5.3.22. The equivalence of (2)\ and (3)\ is analogous to the corresponding equivalence in Thm.~\ref{karp}: one uses the fact that, for the subobject $G$ generated by $Z$, any embedding $G\to Y$ is uniquely determined by its restriction to $Z$.

(3)$\dra$(4) follows from Thm.~\ref{emb-pres}(a).

(4)$\dra$(3): let $\cS$ be a $\lambda$-dense set of $\lambda$-spans between $X$ and $Y$, witnessing that $f$ is a $\lambda$-embedding in $\str(\Sigma)$.  Define $\cS_\emptyset$ analogously to the proof of \ref{karp}(4)$\dra$(3); we claim it verifies that $f$ is a $\lambda$-embedding in $\emb(\Sigma)$.  Indeed, let $X\linto G\in\emb(\Sigma)$ be given, where $G$ is $\lambda$-generated in $\emb(\Sigma)$.  Let $G_0$ be defined as in \ref{karp}; since it is $\lambda$-generated as an object of $\str(\Sigma)$, it gives rise to the diagram
\[ \xymatrix{ X\ar[rr]^f && Y \\
& U \ar[lu]_u\ar[ru]^v \\
& G_0\ar[luu]^g\ar[u]_t } \]
in $\str(\Sigma)$, where $g$ is the composite $X\linto G\linto G_0$ and $X\llla{u}U\llra{v}Y\in\cS$, $ut=g$ and $vt=fg$.  Consider the image factorization
\[ \xymatrix{ X & V\ar@{>->}[l]\ar@{>->}[r] & Y \\
& U \ar[lu]\ar[ru]\ar@{>>}[u] } \]
where $X\linto V\into Y\in\cS_\emptyset$.  $X\linto G$, $X\linto V$ and $V\into Y$ are embeddings of $\Sigma$-structures, whence $t:G_0\to U$ composed with $U\to V$ can be extended to a $\Sigma$-embedding $G\into V$ that satisfies the desired commutativities in $\emb(\Sigma)$.
\end{proof}

\section{Elementary chains}
The Tarski-Vaught theorem states that the union of a chain of $\LL_{\omega\omega}$-elementary embeddings is an $\LL_{\omega\omega}$-elementary embedding.  Here we prove Thm.~\ref{ladder}, various consequences of which are analogues for $\lambda$-embeddings of the Tarski-Vaught theorem.  By the facts established in the previous section, these results do specialize to $\LL_{\infty\lambda}$-elementary embeddings of structures.  It would be interesting to handle Thm.~\ref{ladder} for the logics $\LL_{\kappa\lambda}$, or fragments of, via category-theoretic methods.

Throughout this section, fix a $\lambda$-mono-generated category $\cA$.  We assume that the colimits displayed exist in $\cA$.  For $\lambda$-equivalent objects $X$, $Y$, write $\cS^\ma_\lambda(X,Y)$ for the greatest $\lambda$-dense set of $\lambda$-spans between $X$ and $Y$, cf.\ Remark~\ref{max}(3).

\begin{lemma}  \label{step}
Suppose $X\la U\ra Y\in\cS^\ma_\lambda(X,Y)$ and $f:X\ra X_0$, $g:Y\ra Y_0$ are $\lambda$-embeddings.  Then the span
\[  X_0\llla{f}X\la U\ra Y\llra{g}Y_0   \]
belongs to $\cS^\ma_\lambda(X_0,Y_0)$.
\end{lemma}

Indeed, since $g$ is a $\lambda$-embedding and $U$ is $\lambda$-generated, there exists a $Y\la V\ra Y_0\in\cS^\ma_\lambda(Y,Y_0)$ and $t:U\ra V$ such that in
\[ \xymatrix{ X & Y\ar[rr]^g && Y_0 \\
&& V \ar[lu]_y\ar[ru]^z \\
& U\ar[uu]^u\ar[luu]\ar[ru]_t } \]
$u=yt$ and $gu=zt$.  By Prop.~\ref{compose},
\[  X\la U\ra Y\ra Y_0 = X\la U\ra V\ra Y_0\in\cS^\ma_\lambda(X,Y)\star\cS^\ma_\lambda(Y,Y_0)\subseteq \cS^\ma_\lambda(X,Y_0) \,. \]
A symmetric argument establishes that $X_0\la U\ra Y\in\cS^\ma_\lambda(X_0,Y)$, and the conclusion follows.

\begin{thm}  \label{ladder}
Let $\D$ be a $\lambda$-directed diagram and $F_1,F_2:\D\to\cA$ two functors sending all maps to $\lambda$-embeddings.  Let $\eta:F_1\ra F_2$ be a natural transformation such that $\eta(d):F_1(d)\ra F_2(d)$ is a $\lambda$-embedding for each $d$ in $\D$.  Then the induced $f:\colim F_1\ra\colim F_2$ is a $\lambda$-embedding.
\end{thm}

\begin{proof}
Let $\cS$ be the set of $\lambda$-spans between $\colim F_1$ and $\colim F_2$ that can be factored as
\[  \colim F_1 \llla{k^{(1)}_d} F_1(d) \la U \ra F_2(d) \llra{k^{(2)}_d}\colim F_2  \]
where $d\in\D$, $F_1(d) \la U \ra F_2(d)\in\cS^\ma_\lambda(F_1(d),F_2(d))$ and $k^{(i)}_d$ are parts of the colimit cocone.  We claim that $\cS$ is $\lambda$-dense.  $\cS$ is non-empty since all $k^{(i)}_d$ are monomorphisms.  Fix such an element of $\cS$ and let a mono $g:G\ra\colim F_1$, with $\lambda$-generated $G$, be given.  There exists $d'\in\D$ such that $g$ factors through $k^{(1)}_{d'}:F_1(d')\ra\colim F_1$.  Let $d''\in\D$ be such that $d\ra d''$ and $d'\ra d''$ both exist in $\D$.  In the diagram
\[ \xymatrix{ \colim F_1 &&& \colim F_2 \\
F_1(d'')\ar[u] && V\ar[ll]\ar[r] & F_2(d'')\ar[u] \\
& F_1(d')\ar[ul] & G\ar[l]\ar[u]\ar@{.>}[lluu]_(.6)g  \\
F_1(d)\ar[uu] && U\ar[ll]\ar[r]\ar@/_5mm/[uu] & F_2(d)\ar[uu] } \]
$F_1(d)\ra F_1(d'')$ and $F_2(d)\ra F_2(d'')$ are $\lambda$-embeddings by assumption.  By Lemma~\ref{step},
\[  F_1(d'') \la F_1(d) \la U \ra F_2(d) \ra F_2(d'')  \]
belongs to $\cS^\ma_\lambda(F_1(d''),F_2(d''))$.  Thence $F_1(d'')\la V\ra F_2(d'')\in\cS^\ma_\lambda(F_1(d''),F_2(d''))$ with suitable $G\ra V$ and $U\ra V$ exist.  But that solves the ``back'' direction for $\cS$.  The ``forth'' case is symmetric.

Now we verify the embedding condition.  Given a mono $g:G\ra\colim F_1$ with $\lambda$-generated, let $d\in\D$ be such that $g$ factors as $G\llra{g_0}F_1(d)\llra{k^{(1)}_d}\colim F_1$.  Since $\eta(d)$ is a $\lambda$-embedding, there exist $F_1(d)\la U\ra F_2(d)\in\cS^\ma_\lambda(F_1(d),F_2(d))$ and map $t:G\ra U$
\[ \xymatrix{ \colim F_1\ar[rr]^f && \colim F_2 \\
F_1(d)\ar[u]\ar[rr]^{\eta(d)} && F_2(d)\ar[u] \\
& U\ar[ul]_u\ar[ur]^v  \\
& G\ar[uul]^{g_0}\ar[u]_t } \]
such that $ut=g_0$ and $\eta(d)g_0=vt$.  But $\colim F_1\la F_1(d)\la V\ra F_2(d)\ra\colim F_2$ belongs to $\cS$ by definition, verifying the embedding condition.
\end{proof}

Let $\lambda$-$\emb(\cA)$ be the subcategory of $\cA$ with the same objects, but morphisms the $\lambda$-embeddings of $\cA$.

\begin{cor}   \label{chain}
$(i)$ Let $\D$ be a $\lambda$-directed diagram, and $F:\D\ra\cA$ a functor such that $F(d\ra d')$ is a $\lambda$-embedding for every $d\ra d'$ in $\D$.  Then the colimit cocone of $F$ consists of $\lambda$-embeddings.  $(ii)$ If $\cA$ has $\lambda$-directed colimits of monos, so does $\lambda$-$\emb(\cA)$, created by the inclusion $\lambda$-$\emb(\cA)\inc\cA$.
\end{cor}

\begin{proof}
$(i)$ Pick any $d\in\D$ and let $\D_d$ be the full subdiagram of $\D$ consisting of objects $d'$ such that $d\ra d'$ exists in $\D$.  Apply Thm.~\ref{ladder} with the role of $F_1:\D_d\ra\cA$ played by the constant functor at $F(d)$, $F_2$ being the restriction of $F$ to $\D_d$, and $\eta(d')=F(d\ra d')$.  $(ii)$ The composite $F:\D\ra\lambda$-$\emb(\cA)\inc\cA$ has a colimiting cocone lying in $\lambda$-$\emb(\cA)$.  But this cocone is colimiting in $\emb(\cA)\inc\cA$ as well.  That is, if $\{F(d)\ra X\;|\;d\in\D\}$ is a cocone on $F$ lying in $\lambda$-$\emb(\cA)$, then the induced map $\colim F\ra X$ is a $\lambda$-embedding.  Just apply Thm.~\ref{ladder} with $F_1=F$ and $F_2$ the constant functor at $X$.
\end{proof}

As usual, we'll say ``finitely mono-generated'' instead of ``$\omega$-mono-generated'', and ``finitary embedding'' instead of ``$\omega$-embedding''.  The next corollary states that finitary embeddings are closed under transfinite compositions, i.e.\ smooth chains.

\begin{cor}
Let $\cA$ be a finitely mono-generated category and $\alpha$ an ordinal.  Let $F:\alpha\ra\cA$ be a smooth diagram such that for all $\beta\prec\alpha$, $F(\beta)\ra F(\beta+1)$ is a finitary embedding.  Then $F(0)\ra\colim F$ is a finitary embedding.
\end{cor}

This follows by transfinite induction on $\alpha$, using Prop.~\ref{compo} at successor ordinals and Cor.~\ref{chain}$(i)$ at limit ordinals.

If one replaces $\lambda$-embeddings by $\lambda$-equivalences, the above corollary can certainly fail for uncountable $\lambda$: let $\D$ be the ordered set of countable ordinals $[\omega_0,\omega_1)$ and consider the functor $F:\D\ra\Set$ with $F(\alpha)=\alpha$.  Then $F(\alpha)\sim_{\omega_1}F(\beta)$ for $\alpha,\beta\in\D$, but $\omega_0\not\sim_{\omega_1}\omega_1=\colim F$.  Thm.~\ref{ladder} nonetheless entails an (easy)\ analogue for $\sim_\lambda$.

\begin{cor}
Let $\D$ be a $\lambda$-directed diagram and $F_1,F_2:\D\to\cA$ two functors sending all maps to $\lambda$-embeddings.  Suppose $F_1(d)\sim_\lambda F_2(d)$ for \textsl{at least one} $d\in\D$.  Then $\colim F_1\sim_\lambda\colim F_2$.
\end{cor}

Indeed, $\colim F_1\sim_\lambda F_1(d)\sim_\lambda F_2(d)\sim_\lambda\colim F_2$ by Cor.~\ref{chain}$(i)$ and $\sim_\lambda$ is transitive.  \qed

Given that $\lambda$-$\emb(\cA)$ inherits $\lambda$-directed colimits from $\cA$, it is tempting to ask whether it inherits being accessible or mono-accessible as well.  The answer is seen to be no.  An easy induction shows that for each ordinal $\alpha$, there exists a sentence $\phi(\alpha)\in\LL_{\infty\omega}$ in the signature of the binary relation $<$ such that the only model of $\phi(\alpha)$ (up to isomorphism)\ is $\alpha$.  For any regular $\lambda$, the category $\lambda$-$\emb\big(\str(<)\big)$ thus has a proper class of isolated objects (objects that are not the source or target of any morphism, except the identity).  A mono-accessible category can only have a set of connected components, however.

\end{document}